\documentclass[fleqn,11pt]{article}

\title{On the reconstruction of planar lattice-convex sets\\ from the covariogram}
\author{Gennadiy Averkov\footnote{Institute for Mathematical Optimization, Faculty of Mathematics, Otto-von-Guericke-Universit\"at Magdeburg, Universit\"atsplatz 2, 39106 Magdeburg, Germany, e-mail: averkov@math.uni-magdeburg.de} \, and \, Barbara Langfeld\footnote{Mathematisches Seminar, Christian-Albrechts-Universit\"at zu Kiel, 24098 Kiel, Germany, e-mail: langfeld@math.uni-kiel.de}}

\usepackage{amsfonts,amssymb,amsthm,amsmath,graphicx,enumerate,srcltx}
\usepackage{graphicx,xcolor}

\newcommand{\tikzhide}[1]{#1} \newcommand{\epshide}[1]{}
\newcommand{\pdfhide}[1]{}

\renewcommand{\tikzhide}[1]{}\renewcommand{\pdfhide}[1]{#1}

\usepackage[latin1]{inputenc}

\usepackage{hyperref}%
\hypersetup{%
  pdfauthor={Gennadiy Averkov and Barbara Langfeld},%
  pdftitle={On the reconstruction of planar lattice-convex sets from the covariogram},%
  pdfsubject={},%
  linkcolor=blue, 
  urlcolor=blue, 
  citecolor=blue, 
  breaklinks=true, 
  colorlinks=true, 
  plainpages=false,%
  hypertexnames=false,%
  bookmarksopen=false,%
  bookmarksnumbered=false%
}%

\DeclareMathOperator{\aff}{aff}

\DeclareMathOperator{\bd}{bd}
\DeclareMathOperator{\intr}{int}

\DeclareMathOperator{\cone}{cone}
\DeclareMathOperator{\conv}{conv}
\DeclareMathOperator{\supp}{supp}
\DeclareMathOperator{\vol}{vol}
\newcommand{\bF}{\mathbb{F}}
\newcommand{\compl}{\mathbb{C}}
\newcommand{\bL}{\mathbb{L}}
\newcommand{\card}[1]{\left|#1\right|}

\newcommand{\cT}{\mathcal{T}}
\newcommand{\dotvar}{\,\cdot\,}

\newcommand{\integer}{\mathbb{Z}}

\newcommand{\KK}{\mathcal{K}}

\newcommand{\LCONV}[1]{\KK(\integer^{#1})}
\newcommand{\natur}{\mathbb{N}}

\newcommand{\real}{\mathbb{R}}

\newcommand{\setcond}[2]{\left\{#1\,:\,#2\right\}}
\newcommand{\sprod}[2]{\left<#1,#2\right>}
\newcommand{\term}[1]{\emph{#1}}

\definecolor{update}{HTML}{FFFF00}
\definecolor{oldversion}{HTML}{990099}
\definecolor{commentB}{HTML}{7A91FF} 
\definecolor{commentG}{HTML}{A57AFF} 

\newtheorem{nn}{}[section]
\newtheorem{corollary}[nn]{Corollary}

\newtheorem{remark}[nn]{Remark}
\newtheorem{lemma}[nn]{Lemma}
\newtheorem*{lemma*}{Lemma}

\newtheorem{theorem}[nn]{Theorem}
\newtheorem{question}[nn]{Question}
\newtheorem{problem}[nn]{Problem}

\begin{document}

\maketitle

\begin{abstract}
 A finite subset $K$ of $\integer^d$ is said to be lattice-convex if $K$ is the intersection of $\integer^d$ with a convex set. The covariogram $g_K$ of $K\subseteq \integer^d$ is the function associating to each $u \in \integer^d$ the cardinality of $K\cap (K+u)$. Daurat, G\'erard, and Nivat and independently Gardner, Gronchi, and Zong raised the problem of the reconstruction of  lattice-convex sets $K$ from $g_K$. We provide a partial positive answer to this problem by showing that for $d=2$ and under mild extra assumptions, $g_K$ determines $K$ up to translations and reflections. As a complement to the theorem on reconstruction we also extend the known counterexamples (i.e., planar lattice-convex sets which are not reconstructible, up to translations and reflections) to an infinite family of counterexamples.
\end{abstract}

\newtheoremstyle{itsemicolon}{}{}{\mdseries\rmfamily}{}{\itshape}{:}{ }{}
\newtheoremstyle{itdot}{}{}{\mdseries\rmfamily}{}{\itshape}{.}{ }{}
\theoremstyle{itdot}

\newtheorem*{msc*}{2010 Mathematics Subject Classification} 
\begin{msc*}
  Primary: 52C05, 05B10; Secondary: 16S34.
\end{msc*}

\newtheorem*{keywords*}{Key words and phrases}
\begin{keywords*}
Covariogram, crystallography, diffraction, direct sum, homometric sets, lattice-convex set, Matheron's problem, partial digest problem, quasicrystal, tomography, X-ray.
\end{keywords*}

\section{Introduction}

In the research area of \term{(geometric) tomography} one is usually concerned with the reconstruction of an object from information on the projections and/or sections of that object. Discrete tomography is concerned with the reconstruction of discrete objects and relies on discrete tools (see \cite{MR1722457}, \cite{HerK07}, \cite{MR2160045}) while analytic tomography is concerned with the reconstruction of objects determined by analytic data (see \cite{Gardner-book-2006}). Both research directions are naturally unified within the realm of convex geometry, since convex sets can be of both discrete and analytic nature. Geometric tomography has numerous applications outside mathematics. E.g., one of the many links between tomography and physics is the so-called \term{phase retrieval problem}, that is, the problem of retrieving an unknown object from its \term{diffraction} data (see \cite{Klibanov-Sachs-Tikhonravov-1995}). The diffraction data can be described in mathematical language as an \term{autocorrelation} or a \term{covariogram}.

In this paper we present results on a discrete version of the covariogram problem under convexity assumptions. For presenting the problem we first start with the (standard) covariogram problem in the continuous case. Let $d\in\natur$. By $\vol$ we denote the $d$-dimensional volume. Let $K\subseteq\real^d$ be a \emph{body}, i.e., a bounded set coinciding with the closure of its interior. The (continuous) covariogram of $K$ is the function $g_K:\real^d\to\real$ that sends $u$ to $\vol(K \cap (K+u))$. The covariogram problem is the problem of determining $K$ from $g_K$ within a given class of bodies. Clearly, $g_K$ remains unchanged if we translate $K$ or reflect $K$ with respect to a point (in what follows, the term `reflection' will always mean a point reflection). Thus, the determination of $K$ up to translations and reflections is considered the unique determination. It is not hard to show that $g_K$ is the autocorrelation of the characteristic function of $K$. By this, the covariogram problem is a special case of the phase retrieval problem. Recently there has been much progress on the retrieval of $K$ from $g_K$ within the class of convex bodies (this version of the covariogram problem is usually called \term{Matheron's problem}); see \cite{MR1232748}, \cite{Bianchi02}, \cite{AveBia07}, \cite{Averkov-Bianchi-2009}, \cite{Bianchi-cross-cov-2009}, \cite{Bianchi-tom-cones-2009}, \cite{Bianchi-Gardner-Kiderlen-2011}.

If $K$ is a finite subset of $\integer^d$ we define the (discrete) covariogram of $K$ to be $g_K:\integer^d\to\integer$, $u\mapsto \card{K \cap (K+u)}$; here $\card{\dotvar}$ stands for the cardinality. The problem of retrieving $K$ from $g_K$ is called the discrete covariogram problem. The discrete covariogram problem (for $d=1$ also known as the \term{partial digest problem}) was studied by many authors; see for example \cite{Lemke-Skiena-Smith-2003}, \cite{RoSey82}, \cite{Averkov-detecting-cen-sym-2009}. It turns out that the discrete covariogram problem is a special case of the continuous covariogram problem. In fact, for $K \subseteq \integer^d$, $g_K$ determines $g_{K+[0,1]^d}$ and vice versa. 

A subset $K$ of $\integer^d$ is said to be \emph{lattice-convex} if $K$ is the intersection of $\integer^d$ with a convex set. By $\LCONV{d}$ we denote the class of finite, lattice-convex subsets of $\integer^d$ which affinely span $\real^d$. To the best of our knowledge, the problem of  determination of $K$ from $g_K$ within $\LCONV{d}$ was first addressed in \cite{Daurat-Gerard-Nivat-2005} and \cite{MR2160045}. Both papers also give counterexamples to the uniqueness of reconstruction. In our paper we contribute to the study of this problem. In view of the previous observations, on the one hand, this problem is a special case of the continuous covariogram problem and, on the other hand, a discrete version of Matheron's problem. We emphasize that recently the covariogram problem for lattice-convex sets attracted attention of researchers working in \term{X-ray crystallography} and the theory of \term{quasicrystals}; see~\cite{BaakeGrimm}. (We refer to \cite{quasicrystals-primer} and \cite{Moo00} for further information on the latter topics.) The discrete covariogram problem is also related to the \term{discrete X-ray problem} (i.e., the problem of reconstruction from discrete X-ray pictures). In \cite[Theorem~4.1]{MR2160045}  it was proved that sets $K, L \in \LCONV{2}$ have equal covariograms if and only if for every direction the discrete X-ray pictures of $K$ and $L$ are (so-called) rearrangements of each other. Strong positive results on the discrete X-ray problem for the class $\LCONV{2}$ were obtained in \cite{MR1376547}.

In this paper we address the covariogram problem for  $\LCONV{2}$ and provide both positive and negative results. One can see that for $K \in \LCONV{2}$ the directions of the (discrete) edges of $K$ are determined by $g_K$. Our positive result (Theorem~\ref{thm:main-thm}) implies that, for a given collection of prescribed edge directions,  a set $K \in \LCONV{2}$ is determined by $g_K$ if all edges of $K$ have sufficiently large cardinality and the difference between cardinalities of parallel edges is either zero or sufficiently large. The proof of the positive result suggests that, under the given sufficient conditions, the retrieval of $K$ from $g_K$ is analogous to the retrieval of polygons from the continuous covariogram (the problem solved in \cite{MR1232748}, see also \cite{Bianchi02} for a short solution). Thus, the discrete nature of the problem is only revealed when the above sufficient conditions are violated. As examples in \cite[Section~4]{Daurat-Gerard-Nivat-2005} and \cite[p.~402]{MR2160045} show, $K \in \LCONV{2}$ is not always determined by $g_K$. Therefore, as a complement to the positive result, we also study those $K \in \LCONV{2}$ which are not determined by $g_K$. All known $K \in \LCONV{2}$ not determined by $g_K$ can be represented as a direct sum $K = S \oplus T$, where $S \subseteq \integer^2$ and $T \in \LCONV{2}$ are not centrally symmetric and $T$ can be covered by two distinct parallel lines. Thus, we are concerned with the characterization of all $K$ having the above form. It turns out that the class of all such $K$ is quite narrow (the convex hull of $S$ is a certain polygon with at most six edges and $T$ is also of a certain particular form); this is our second main result which is contained in Theorem~\ref{thm:homometry:from:lat:wid:one}. The complete classification of all $K \in \LCONV{2}$ determined by $g_K$ remains an open question.

The paper is organized as follows. In~Section~\ref{sec:main-results} we introduce necessary notation, present the main results (Theorem~\ref{thm:main-thm} and Theorem~\ref{thm:homometry:from:lat:wid:one}), and state consequences and open questions. Theorem~\ref{thm:main-thm} and Theorem~\ref{thm:homometry:from:lat:wid:one} are proved in Sections~\ref{sec:mainproof} and \ref{sec:HomometricPairs}, respectively.

\section{Results and open questions}\label{sec:main-results}

For formulating our main results we need further notations and terminology. The origin in $\real^d$ will be denoted by $o$. If $u \in \integer^d \setminus \{o\}$, then $\gcd(u)$ stands for the greatest common divisor of the components of $u$. For $u_1,u_2 \in \integer^2$ we denote by $\det(u_1,u_2)$ the determinant of the matrix with columns $u_1$ and $u_2$ (in this order). The interior, boundary, affine hull, and conical hull operations will be denoted by $\intr$, $\bd$, $\aff$, and $\cone$, respectively. The convex hull of $K\subseteq \real^d$ will be denoted by $\overline{K}$.\footnote{We decided to use the notation $\overline K$ instead of the standard notation $\conv K$, because it allows us to write formulas in a more compact way.} The closure of the set $\setcond{u\in\real^d}{g_K(u)>0}$ is called the  \term{support} of $g_K$ and is denoted by $\supp g_K$. As usual, for $S,T\subseteq \real^d$ the \term{(Minkowski) sum} $S+T$ is defined to be $\setcond{s+t}{s \in S \,\land\, t \in T}$; for $t \in \real^d$ we also write $S+t$ instead of $S+\{t\}$. Further, we put $-T:=\setcond{-t}{t\in T}$. The notations $t-S$ and $S-t$ are introduced in a similar fashion. For $R \subseteq \real$ and $S \subseteq \real^d$ we use the notation $R S = \setcond{r s}{r \in R \,\land\, s \in S}$.

For $K\subseteq\real^d$ we put $[K]:=\setcond{K+t}{t\in\real^d}\cup \setcond{-K+t}{t\in\real^d}$, which is the set of translations and reflections of $K$. A pair $(K,L)$ of subsets of $\integer^d$ is said to be \term{homometric} if $g_K=g_L$ and \term{nontrivially homometric} if additionally $[K] \ne [L]$.

Let us introduce the notion of lattice-convexity with respect to an arbitrary lattice. A subset $K$ of a lattice $\bL \subseteq \real^d$ is said to be \term{lattice-convex} (with respect to $\bL$) if $K$ is the intersection of $\bL$ and a convex set. If $\bL$ is clear from the context, we will drop the qualifier `with respect to $\bL$'. By $\KK(\bL)$ we denote the class of finite, lattice-convex subsets $K$ of $\bL$ with $\aff K= \real^d$.

We use a few standard notations and definitions of convex geometry (in some of the cases, carried over to the case of lattice-convex sets); for details see \cite{MR1216521}. If $K$ is a subset of $\real^d$, the \term{support function} $h(K,\dotvar) : \real^d \rightarrow \real \cup \{\infty\}$ of $K$ is defined by $h(K,u) := \sup \setcond{\sprod{x}{u}}{x \in K}$, where $\sprod{\dotvar}{\dotvar}$ is the standard scalar product in $\real^d$. If $K \subseteq \real^d$ is compact, then the \term{support set} of $K$ in direction $u \in \real^d$ is defined by \[F(K,u):=\setcond{x \in K}{\sprod{x}{u} = h(K,u)}.\] Clearly, if $K \in \LCONV{d}$ and $u \in \real^d$, then $F(K,u)$ is lattice-convex. For support sets we have
\begin{equation} \label{eq:sum-of-faces}
F(K+L,u) = F(K,u) + F(L,u)
\end{equation}
for bounded compact sets $K, L \subseteq \real^d$ and each $u \in\real^d$.

Some terminology from classical convexity theory can be directly carried over to the setting of lattice-convex sets. If $K \in \LCONV{2}$, then a support set $F(K,u)$ with $u \in \real^2 \setminus \{o\}$ and that contains exactly one element will be called a \term{vertex} of $K$; a support set $F(K,u)$ with $u \in \real^2\setminus\{o\}$ and with more than one element will be called an \term{edge} of $K$ with \term{outer normal} $u$.
We introduce
\[
	U(K) = \setcond{u \in \integer^2 \setminus \{o\}}{u \ \text{is an outer normal to an edge of $K$ and $\gcd(u)=1$}}.
\]
	Vectors $u \in \integer^2 \setminus \{o\}$ with $\gcd(u) =1$ are called \term{primitive}.

To measure the number of lattice points on the edges and the difference of parallel edges of $K\in\LCONV{2}$ we introduce
\begin{align*}
 m'(K)&:=\min\setcond{|F(K,u)|}{u\in U(K)},\\
 m''(K)&:=\min\setcond{|F(K,u)|-|F(K,-u)|+1}{u\in \integer^2\setminus\{o\}\,\land\,|F(K,u)|>|F(K,-u)|>1},\\
 m(K)&:=\min \{m'(K),m''(K)\},
\end{align*}
where as usual $\min\emptyset=\infty$. Observe that $m''(K)=\infty$ if and only if $K$ is centrally symmetric, and that $m(K)=m'(K)$ if $K$ has no parallel edges of different lengths.

If $U$ is a finite set of vectors in $\real^2$ which linearly spans $\real^2$, then the \term{discrepancy} $\delta(U)$ of $U$ is
\[\delta(U):=\frac{\max D(U)}{\min D(U)},\]
where $ D(U):= \setcond{|\det(u_1,u_2)|}{u_1, u_2 \in U} \setminus \{0\}$. For $K\in\LCONV{2}$ we define $\delta(K):=\delta(U(K))$. The geometric meaning of the discrepancy of a system of vectors can be easily visualized in terms of cardinalities of lattice-convex sets. For two primitive vectors $u, v\in\integer^2$ the number $|\det(u,v)|$ is equal to the number of lattice points in the half-open parallelogram $[0,1)u+[0,1)v$; this can be derived, e.g., from \term{Pick's formula}. In this sense, for primitive $u$ and $v$, $|\det(u,v)|$ is a measure for `how densely' $\cone \{u,v\}$ is filled with lattice points, relatively to `how densely' the border of this cone is filled. The discrepancy $\delta(U)$ then can be understood as a number measuring the `unbalancedness' of this kind of density of the cones that can be formed by pairs of nonparallel vectors from $U$.

Now we are ready to formulate our first main result.

\begin{theorem} \label{thm:main-thm} 
 Let $K, L \in \LCONV{2}$ be sets with $g_K=g_L$.
 Then $m'(K)$, $m''(K)$, $m(K)$, $U(K)\cup U(-K)$ and $\delta(K)$ are determined by $g_K$. Moreover, if $m(K) \ge \delta(K)^2+\delta(K)+1$, then $[K]=[L]$; hence $g_K$ determines $K$ within the class $\LCONV{2}$ up to translations and reflections.
\end{theorem}

The proof of Theorem~\ref{thm:main-thm} will be given in Section~\ref{sec:mainproof}. It is based on the proof ideas of results for continuous covariograms presented in \cite{Bianchi02}, \cite{MR1938112}, \cite{AveBia07}, \cite{Averkov-Bianchi-ArXiv-2007}. Most parts of our proof of Theorem~\ref{thm:main-thm} can be rendered into an algorithmic form. A detailed analysis of algorithmic issues would certainly be interesting and valuable, but it would go beyond the scope of this paper.

\begin{remark}
	Note that $\delta(K) \le 2n^2$ whenever $U(K) \subseteq \{-n,\ldots,n\}^2$ and $n \in \natur$. Thus, the assumption on $m(K)$ in Theorem~\ref{thm:main-thm} can be replaced by the less technical but somewhat stronger assumption $m(K) \ge 4 n^4 + 2 n^2+1$.
\end{remark}

Since examples in \cite{MR2160045} and \cite{Daurat-Gerard-Nivat-2005} show that $K \in \LCONV{2}$ is not always reconstructible from $g_K$, we also study nonreconstructible sets $K \in \LCONV{2}$. That is, we work towards a description of nontrivially homometric pairs of sets lying in $\LCONV{2}$.

We first present known information on the related question on the structure of homometric pairs of finite subsets of $\integer^d$. In \cite{RoSey82} it is shown that homometry for subsets of $\integer^d$ can be explained by considering the reconstruction from the autocorrelation within the group ring $\bF[\integer^d]$, where $\bF$ is $\integer$, $\real$, or $\compl$. Let us give more details. Let $Z:=Z_1,\ldots,Z_d$ be a sequence of independent commuting indeterminates. Given $u=(u_1,\ldots,u_d) \in \integer^d$, by $Z^u$ we denote the expression $Z_1^{u_1} \cdots Z_d^{u_d}$. Then $\bF[\integer^d]$ is the ring of all expressions $f :=\sum_{u \in \integer^d} a_u Z^u$ such that $a_u \in \bF$ is equal to zero for all but finitely many $u$; the ring addition and multiplication are introduced in the same way as in polynomial rings. For $f$ as above we define $f^\ast := \sum_{u \in\integer^d} \overline{a_u} Z^{-u}$, where $\overline{a_u}$ is the complex conjugate of $a_u$. The \term{autocorrelation} of $f$ is then defined to be $f f^\ast$. In \cite[Theorem~2.3]{RoSey82} it is shown that $f_1, f_2 \in \bF[\integer^d]$ are homometric (i.e., have equal autocorrelations) if and only if $f_1 = Z^{u} h_1 h_2$ and $f_2 = c Z^{v} h_1 h_2^\ast$ for some $u, v \in \integer^d$, $c \in \bF$ with $|c|=1$ and $h_1, h_2 \in \bF[\integer^d]$. (The proof is based on the fact that $\bF[\integer^d]$ is a unique factorization domain.) Since every finite subset $K$ of $\integer^d$ can be identified with its \term{generating function} $f_K:= \sum_{u \in K} Z^u$ and since $f_K f_K^\ast$ provides the same information as $g_K$, the latter result can also be used for homometric pairs of finite subsets of $\integer^d$.

As the problem of reconstruction of a finite $K \subseteq \integer^d$ from $g_K$ is more geometric than the analogous problem for $\bF[\integer^d]$, one might expect a more geometric reason of homometry. This is sometimes the case. The explanation will require the notion of direct (Minkowski) sums. We call the Minkowski sum $K=S+T$ of two lattice sets $S,T\in\integer^d$ \emph{direct}, if for each $s,s'\in S$ and $t,t'\in T$ the equality $s+t=s'+t'$ implies $s=s'$ and $t=t'$. In this case we write $K=S\oplus T$. It is easy to see that if the sum of $S$ and $T$ is direct, then the sum of $S$ and $-T$ is also direct. One also readily verifies that if $K,L\in\integer^d$ satisfy $K=S\oplus T$ and $L=S\oplus(-T)$ for suitable $S,T\in\integer^d$, then $g_K=g_L$. Moreover, in this case one has $[K]\neq[L]$ if and only if both $S$ and $T$ are not centrally symmetric. So direct Minkowski sums provide a way to construct examples of pairs of nontrivially homometric lattice sets. We say that the homometric pairs $(K,L)$ obtained this way are \term{geometrically constructible}. For $K, L, S, T$ as above the relations $K=S \oplus T$ and $L=S \oplus (-T)$ carry over to $f_K = f_{S \oplus T} = f_S f_T$ and $f_{S \oplus (-T)} = f_S f_{-T} = f_S f_T^\ast$. This provides a link to the result for $\bF[\integer^d]$. In \cite{RoSey82} a homometric pair of subsets of $\integer^d$ (for $d=1$) was given which is not geometrically constructible. This justifies the extension from the class of finite subsets of $\integer^d$ to $\bF[\integer^d]$.

We wish to study the retrieval of $K$ from $g_K$ for $K \in \LCONV{d}$. Taking into account the presented information on $\bF[\integer^d]$, the following problem can be formulated:

\begin{problem} \label{factors:of:g_K}
	Let $d \ge 2$ and let $\bF$ be $\integer$, $\real$ or $\compl$. Describe all $h \in \bF[\integer^d]$ such that $h$ is a factor of $f_K$ for some $K \in \LCONV{d}$.
\end{problem}

Let us briefly discuss the analogous problem for $d=1$. The problem resembles to the study of the factors of the polynomial $\sum_{k=0}^{n-1}Z^k=(Z^n-1)/(Z-1)$ for some $n\in\natur$. For $\bF=\compl$, the irreducible factors correspond to just the roots of unity (except  $1$), and for $\bF=\real$ one can describe the factors with the results for the complex case. For $\bF=\integer$ the problem is well-studied, see \cite[Ch.~VI, \S3]{MR1878556}; the corresponding factors can be determined with the help of the so-called \term{cyclotomic polynomials}.

Also, the following question arises naturally:

\begin{question} \label{quest:lconv:homometric:non-geom}
 Do there exist nontrivially homometric pairs $(K,L)$ of sets in $\LCONV{2}$ which are not geometrically constructible?
\end{question}

We performed an exhaustive computer search which showed that all homometric pairs of sets $\LCONV{2}$ which are contained in $\{1,\ldots,6\} \times \{1,\ldots,5\}$ are geometrically constructible. Our computer search also covered the homometric pairs presented in \cite{MR2160045} and \cite{Daurat-Gerard-Nivat-2005}, up to affine automorphisms of $\integer^2$. The fact that the example in \cite{MR2160045} is geometrically constructible was noticed in \cite[pp.~280-281]{Benassi-Bianchi-DErcole-2010}.

Another natural research direction is the study of geometrically constructible homometric pairs from $\LCONV{2}$. A partial information is provided by Theorem~\ref{thm:homometry:from:lat:wid:one} given below. A set $T \in \LCONV{2}$ is said to have \term{lattice width $1$} if $T$ can be covered by two distinct parallel lines. Theorem~\ref{thm:homometry:from:lat:wid:one} characterizes lattice-convex direct sums $S \oplus T$ with $S \subseteq \integer^2$ and with $T\in \LCONV{2}$ having lattice width $1$.

\begin{theorem} \label{thm:homometry:from:lat:wid:one} 
Let $k,\ell$ be integers with $k>\ell\geq 0$. We define $T:=(\{0,\ldots,k\} \times \{0\})\cup(\{0,\ldots,\ell\} \times \{1\})$, the vectors $w_1:= (-k-1,1)$, $w_2:= (\ell + 1, 1)$, and the lattice $\bL:=\integer w_1+\integer w_2$. Let $S$ be a set with $o \in S\subseteq \integer^2$. Then the following conditions are equivalent:
 \begin{enumerate}[(i)]
 \item \label{item:direct-and-convex} The sum of $S$ and $T$ is direct and lattice-convex.
 \item \label{item:S-is-special} $S$ is lattice-convex with respect to $\bL$ and $\overline{S}$ is a polyhedron (in $\real^2$) such that
	\begin{itemize}
		\item every edge of $\overline{S}$ is parallel to $w_1$ or $w_2$ if $k > \ell+1$ and
		\item every edge of $\overline{S}$ is parallel to $w_1$, $w_2$, or $w_1+w_2$ if $k = \ell +1$.
	\end{itemize}
 \end{enumerate}
\end{theorem}

The assumption $o \in S$ is not really restrictive, since Condition~\eqref{item:direct-and-convex} is invariant under translations of $S$. Thus, assuming $o \in S$,  Condition~\eqref{item:S-is-special} can be formulated in a less technical form. Condition~\eqref{item:S-is-special} is invariant under the change of $S$ by $-S$. Thus, if \eqref{item:S-is-special} holds, then $S \oplus T$ and $S \oplus (-T) = - ((-S) \oplus T)$ are both lattice-convex. If $k > \ell+1$ and $S$ is a finite set satisfying Condition~\eqref{item:S-is-special}, then $\overline{S}$ is a parallelogram or a line segment (and by this, centrally symmetric). Thus, as a direct consequence of Theorem~\ref{thm:homometry:from:lat:wid:one}  we obtain:

\begin{corollary}\label{cor:counterexa}
	Let $S$, $T$, $k$ and $\ell$ be as in the formulation of Theorem~\ref{thm:homometry:from:lat:wid:one}, let $S$ be finite and the sum of $S$ and $T$ be direct. Then the following conditions are equivalent: 
	\begin{enumerate}[(i)]
		\item $(S \oplus T, S \oplus (-T))$ is a nontrivially homometric pair of sets in $\LCONV{2}$.
		\item All of the following properties are satisfied:
		\begin{itemize}
			\item $k=\ell+1$,
			\item $S$ is lattice-convex with respect to $\bL$,
			\item $\overline{S}$ is a polyhedron (in $\real^2$) with every edge parallel to $w_1$, $w_2$, or $w_1+w_2$,
			\item $S$ is not centrally symmetric.
		\end{itemize}
	\end{enumerate}
\end{corollary}

 Corollary~\ref{cor:counterexa} describes (up to affine automorphisms of $\integer^d$) all pairs of nontrivially homometric sets that can be written as $(S\oplus T, S\oplus (-T))$ with a lattice-convex summand $T$ of lattice width $1$. Thus, on the one hand,  Corollary~\ref{cor:counterexa} shows that there are infinitely many nontrivially homometric pairs of sets in $\LCONV{2}$ (see Figure~\ref{fig:counterexa} for an illustration). On the other hand, under the given restrictions on $T$, these homometric pairs are `sporadic' in the sense that $\overline{S}$ has at most six edges and very restricted normal vectors and the shape of $T$ is also very specific (see the condition $k=\ell+1$ in Corollary~\ref{cor:counterexa}).
\begin{figure}
 \centering \tikzhide{\tikzCounterexa}%
 \pdfhide{\includegraphics{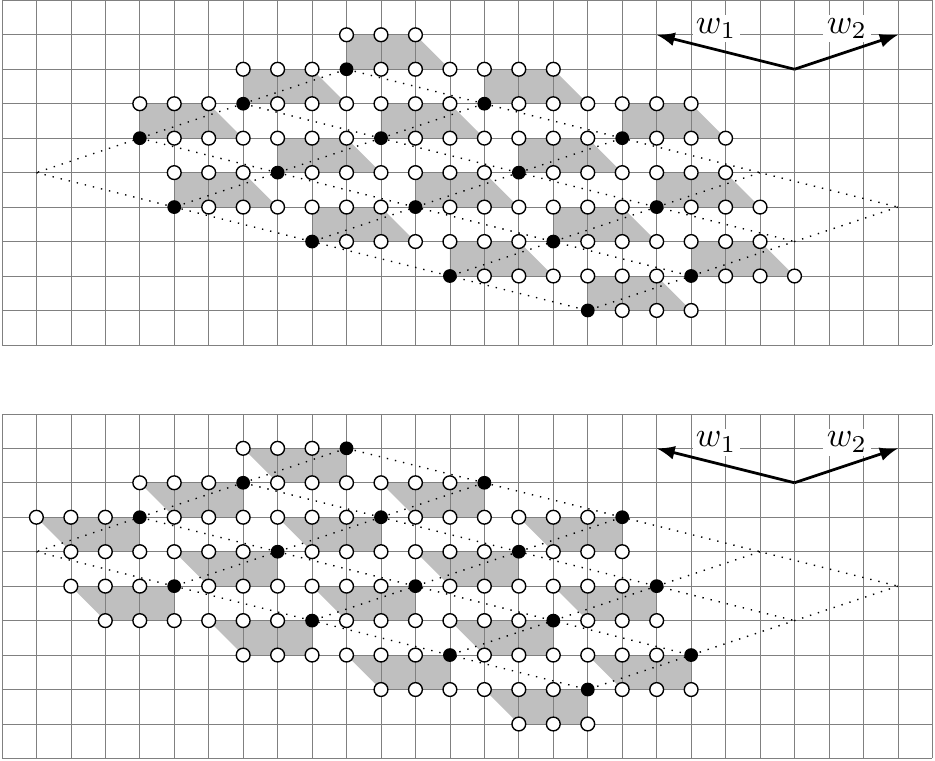}}
 \parbox[t]{0.9\textwidth}{%
 \caption{Corollary~\ref{cor:counterexa} generalizes the counterexample from \cite{MR2160045}. This picture shows $S\oplus T$ (above) and $S\oplus(-T)$ (below). The elements of $S$ are drawn as black points and the convex hulls of the translates of $T$ and $-T$ are indicated by gray polygons.}
 \label{fig:counterexa}}
\end{figure}%

We note that, up to affine automorphisms of $\integer^2$, all homometric pairs found by our computer search are covered by Corollary~\ref{cor:counterexa}. Thus, we ask the following question:
\begin{question}
 Do there exist nontrivially homometric pairs which are geometrically constructible and do not coincide (up to affine automorphisms of $\integer^2$) with the homometric pairs in Corollary~\ref{cor:counterexa}?
\end{question}

\begin{remark}\label{rem:Baake}
 In~\cite{BaakeGrimm} Baake and Grimm used the example from~\cite{MR2160045} to construct a pair of `distinct homometric model sets'. Loosely speaking, these are mathematical descriptions of quasicrystals that differ from each other in many significant ways, but that cannot be distinguished by their diffraction data. Our Corollary~\ref{cor:counterexa} can be used to extend the result of~\cite{BaakeGrimm} to an infinite series of pairs of distinct homometric model sets.
\end{remark}

\section{Proof of Theorem~\ref{thm:main-thm}}
\label{sec:mainproof}

The proof of Theorem~\ref{thm:main-thm} is organized as follows. First we show that $g_K$ contains (up to translations and reflections) the `local boundary information' on $K$, namely edges (Lemma~\ref{lem:edges-up-to-flips}) and normal cones (Lemma~\ref{lem:all-cones}). These arguments are discrete equivalents of some of the arguments given in \cite{Bianchi02}. The key arguments contained in Lemma~\ref{lem:Bianchi-main} show that the `local boundary information' can be assembled in the unique way to reproduce the boundary of $\overline{K}$ up to translations and reflections. One the one hand, the proof of Lemma~\ref{lem:Bianchi-main} is related to arguments in \cite{Bianchi02} and, on the other hand, to the arguments on the so-called \term{capturing arcs} used in \cite{MR1938112} and \cite{AveBia07} (see also \cite{Averkov-Bianchi-ArXiv-2007} which is a full-length version of \cite{AveBia07}). Somewhat more precisely, if we have a pair of arcs of $\overline{K}$ which is symmetric (with respect to reflection) and this pair is adjacent to a nonsymmetric pair of arcs of $\overline{K}$, then, using certain values of $g_K$, the symmetric and the nonsymmetric pair can be `glued together'. This `gluing procedure' is analogous to a continuous procedure used in \cite[Lemma~3 and the proof of Theorem~1]{AveBia07} (see also \cite[the proofs of Lemma~9 and Theorem~1]{Averkov-Bianchi-ArXiv-2007}). The mentioned arguments from \cite{AveBia07} are refined versions of the arguments from \cite[Definition~4.1, Lemma~4.2, Proposition~4.3]{MR1938112}. Of course, in our setting, we have to adjust the arguments of~\cite{Bianchi02}, \cite{MR1938112} and \cite{AveBia07} (or \cite{Averkov-Bianchi-ArXiv-2007}) to match the discrete case. 

If $K \subseteq \real^2$, then $DK:=K+(-K)$ is said to be the \term{difference set} of $K$. If $K$ is a convex body, then $DK$ is called the \term{difference body} of $K$. It is known and easy to show that for each finite $K \subseteq \integer^d$ one has $\supp g_K = DK$ and $\supp g_{\overline K} = D\overline{K}=\overline{DK}$. If $v$ is a vertex of $K$, we define the \term{normal cone} $N(K,v):= \setcond{u \in \real^d}{v \in F(K,u)}$ and the \term{supporting cone} $S(K,v) := \cone (K-v)$ of $K$ at $v$ (see also \cite[p.\,70]{MR1216521}).
\begin{lemma}\label{lem:edges-up-to-flips}
 Let $K,L \in \LCONV{2}$ with $g_K=g_L$. Then for each $u \in \integer^2 \setminus \{o\}$ there exists $L'\in[L]$  satisfying
	\begin{align}
		F(K,u)&=F(L',u) & &\text{and} &
		F(K,-u)&=F(L',-u). \label{eq:edges-up-to-flips}
	\end{align}
In particular, $m'(K)=m'(L)$, $m''(K)=m''(L)$, $m(K)=m(L)$, $U(K)\cup U(-K)=U(L)\cup U(-L)$, and $\delta(K)=\delta(L)$.
\end{lemma}
\begin{proof}
 Since $DK = \supp g_K$, the set $DK$ is determined by $g_K$. We reconstruct $F(K,u) \cup F(K,-u)$ by analyzing the values of $g_K$ on $\aff F(DK,u) \cap \integer^2$. 

\emph{Case 1: $F(DK,u)$ is an edge of $DK$.} Then at least one of $F(K,u)$ and $F(K,-u)$ is an edge of $K$, see~\eqref{eq:sum-of-faces}. We consider a nonsingular affine transformation $\tau$ mapping $\integer$ onto $\aff F(DK,u) \cap \integer^2$ (that is, we enumerate the elements of $\aff F(DK,u) \cap \integer^2$ by elements of $\integer$ in a consecutive way). Then there exist $k_1,\ldots, k_4 \in \integer$ with $k_1 \le k_2 \le k_3 \le k_4$ such that $g_K(\tau(k)) = 0$ for $k < k_1$ and $k > k_4$ and $g_K(\tau(k))$ is the maximum of $g_K$ on $\aff F(DK,u) \cap \integer^2$ for $k_2 \le k \le k_3$. It can be seen that $F(K,u) \cup F(K,-u)$ is a translation  or a reflection of $\setcond{\tau(k)-\tau(k_2)}{k_1\leq k\leq k_2} \cup \setcond{\tau(k)}{k_1\leq k\leq k_3}$, see Figure~\ref{fig:edges-up-to-flips} for an illustration.

\emph{Case 2: $F(DK,u)$ is a vertex of $DK$.} Then both $F(K,u)$ and $F(K,-u)$ are vertices of $K$. In analogy to the argumentation above one shows that $F(K,u) \cup F(K,-u)$ is a translation of $\{o\}\cup F(DK,u)$.
\end{proof}

\begin{figure}
 \centering \tikzhide{\tikzEdgesUpToFlips}
 \pdfhide{\includegraphics{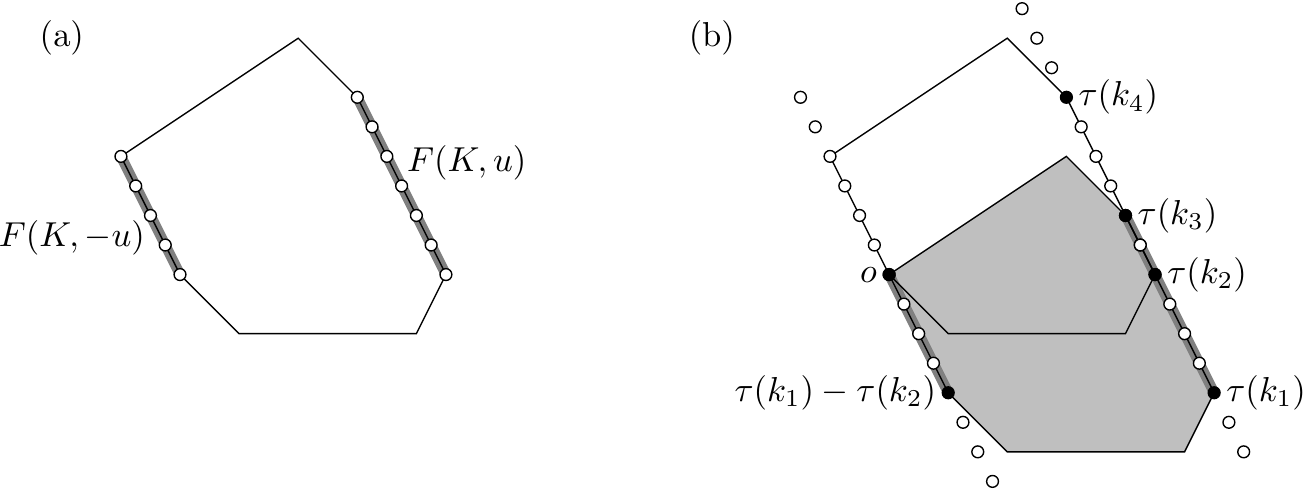}}
 \parbox[t]{0.9\textwidth}{%
 \caption{Illustration of the proof of Lemma~\ref{lem:edges-up-to-flips}, Case 1. (a) A pair of parallel edges of $K$. (b) The reconstruction of this pair from the covariogram.}
 \label{fig:edges-up-to-flips}}
\end{figure}%

In view of Lemma~\ref{lem:edges-up-to-flips} we see that for $K \in \LCONV{2}$ the parameters $m'(K)$, $m''(K)$, $m(K)$ and set $U(K) \cup U(-K)$ are determined by $g_K$, which proves the first statement of Theorem~\ref{thm:main-thm}.

Lemma~\ref{lem:k-exists} given below will be used to ensure that certain triangles are contained in homometric lattice-convex sets. 
\begin{lemma}\label{lem:k-exists}
  Let $K \in \LCONV{2}$ and let $m:=m(K)$ and $\delta:=\delta(K)$ satisfy $m \ge \delta^2+\delta+1$. Then there exists $k\in\natur$ such that $\delta \le k \le {(m-1)}/{\delta}$. Moreover, if $u,v,w\in U(K)\cup U(-K)$ are mutually nonparallel, then $ku=\alpha v+\beta w$ for some $\alpha,\beta\in\real$ that satisfy $1\leq |\alpha|\leq m-1$ and $1\leq|\beta|\leq m-1$. The same holds if $u,v,w\in\integer^2\setminus\{o\}$ are mutually nonparallel, primitive vectors that are parallel to edges of $\overline{K}$.
\end{lemma}

\begin{proof} The first assertion of Lemma~\ref{lem:k-exists} follows from $(m-1)/\delta-\delta\geq 1$. Let $u,v,w\in U(K)\cup U(-K)$ be mutually nonparallel. There exist $\alpha,\beta\in\real\setminus\{0\}$ with $ku=\alpha v+\beta w$. Employing \term{Cramer's rule}, we get
\begin{align*}
 {|\alpha|}/{k} & ={|\det(u,w)|}/{|\det(v,w)|}, & {|\beta|}/{k}
 & ={|\det(v,u)|}/{|\det(v,w)|}.
\end{align*}
Using the definition of $\delta$ and the choice of $k$, this implies that $1\leq {k}/{\delta}\le |\alpha| \le k \delta\leq m-1$ and, similarly, $1\leq {k}/{\delta}\le |\beta| \le k \delta\leq m-1$. If $U$ is the set of primitive vectors parallel to edges of $\overline{K}$, then
$\delta(K)=\delta(U)$. This implies the second assertion.
\end{proof}

\begin{lemma} \label{lem:all-cones}
Let $K, L \in \LCONV{2}$ be such that $g_K=g_L$. Let $m(K) \ge \delta(K)^2+\delta(K)+1$. Let $v$ be a vertex of $K$ and let $u \in \integer^2 \setminus \{o\}$ satisfy $F(K,u)=\{v\}$. Then there exists $L'\in[L]$ such that $F(K,u)=F(L',u)$, $F(K,-u)=F(L',-u)$, and $S(K,v)=S(L',v)$. Moreover, if $F(K,-u)=\{v'\}$ for some vertex $v'$ of $K$, then $L'$ can be chosen to additionally satisfy $S(K,v')=S(L',v')$.
\end{lemma}

\begin{proof} Put $m:=m(K)$ and $\delta:=\delta(K)$.

 \emph{Case 1: $F(K,-u)$ is an edge of $K$.}
 Due to Lemma~\ref{lem:edges-up-to-flips} we find $L'\in[L]$ such that $F(K,u)=F(L',u)$, $F(K,-u)=F(L',-u)$. It remains to show that $S(K,v)=S(L',v)$. Let $w_1,\ldots,w_4 \in \integer^2$ be vectors with $\gcd(w_i)=1$ for each $i\in\{1,\ldots,4\}$ such that if $E$ is an edge of $K$ or $L'$ with $v \in E$, then for some $i\in\{1,\ldots,4\}$ the point $v+w_i$ is the element of $E \setminus \{v\}$ closest to $v$. We argue by contradiction and assume that $S(K,v) \ne S(L',v)$. After possibly interchanging $K$ with $L'$ and reordering $w_1,\ldots,w_4$, we assume that $S(K,v)= \cone \{w_1,w_2\}$ and $w_1 \not\in S(L',v)$. Let $w \in \integer^2$ with $\gcd(w)=1$ be parallel to $F(K,-u)$ and directed in such a way that $w_1\in\cone\{w_i,-w\}$ for $i\in\{2,3,4\}$. See Figure~\ref{fig:all-cones} for an illustration of the situation and the following arguments.

 \begin{figure}
 \centering \tikzhide{\tikzAllCones}%
 \pdfhide{\includegraphics{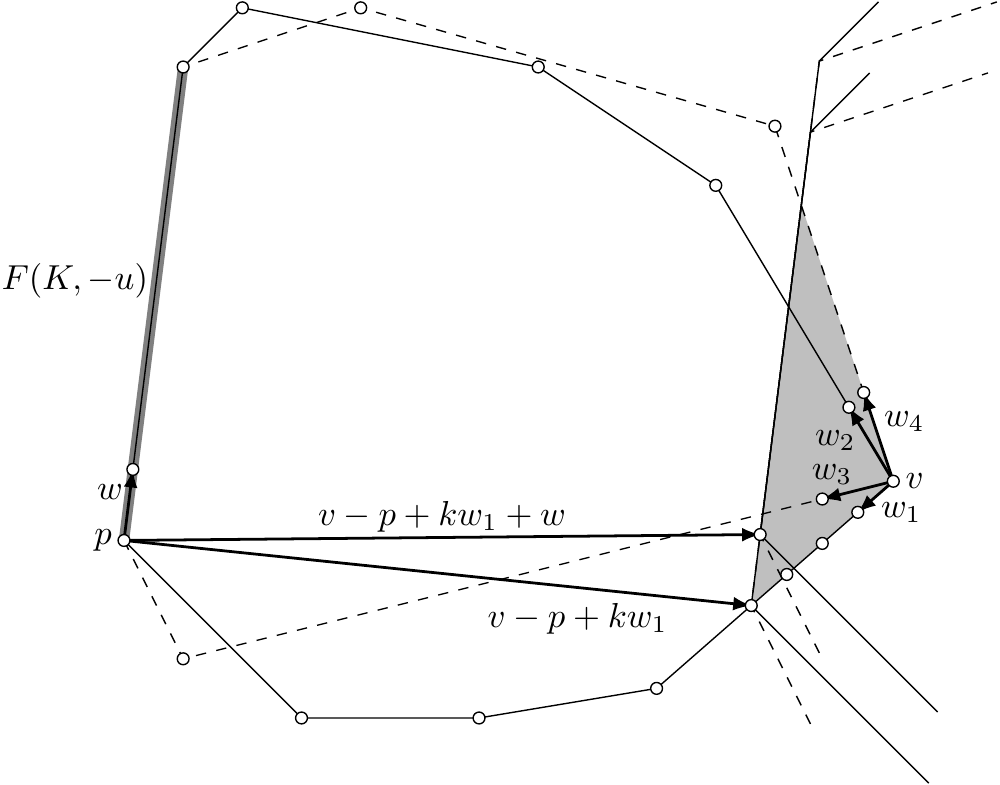}}
 \parbox[t]{0.9\textwidth}{%
 \caption{Illustration of the proof of Lemma~\ref{lem:all-cones}, Case 1. The boundary of $\overline{K}$ and its translations are drawn with a solid line. The boundary of $\overline{L'}$ and its translations are drawn with a dashed line. }\label{fig:all-cones}}
\end{figure}%

By Lemma~\ref{lem:k-exists} there exists $k\in\natur$ with $k\leq m-1$ and there exist real numbers $\alpha_2,\alpha_3,\alpha_4,\beta_2,\beta_3,\beta_4$ such that $k w_1 = \alpha_i w_i - \beta_i w$ and $1\leq \alpha_i\leq m-1$ and $1\leq\beta_i\leq m-1$ for each $i\in\{2,3,4\}$. Let $p$ be the vertex of $F(K,-u)=F(L',-u)$ such that $p+w$ belongs to $F(K,-u)=F(L',-u)$. Since $\alpha_i,\beta_i \le m-1$ for each $i\in\{2,3,4\}$ and by definition of $m=m(K)$, the sets $\overline{K} \cap (\overline{K}+ v - p + k w_1)$ and $\overline{L'} \cap (\overline{L'}+ v - p + k w_1)$ are triangles whose edges are parallel to the vectors from $\{w,w_1,\ldots,w_4\}$. The fact that $\beta_3,\beta_4 \ge 1$ implies that $v+kw_1+w$ is not contained in the interior of $\overline{L'}$. So, by construction, one has
\begin{align*}
 v + k w_1 \not\in \overline{K} \cap (\overline{K}+ v - p + k w_1+w)
 & \varsubsetneq \overline{K} \cap (\overline{K}+ v - p + k w_1) \ni v + k w_1,\\
 \overline{L'} \cap (\overline{L'}+ v - p + k w_1 + w) & = \overline{L'} \cap
 (\overline{L'}+ v - p + k w_1).
\end{align*}
It follows that $g_K(v - p + k w_1+w) < g_K(v-p+k w_1)$ and $g_{L'}(v - p + k w_1+w) = g_{L'}(v-p+k w_1)$ and by this $g_K \ne g_{L'}$, a contradiction.

\emph{Case 2: $F(K,-u)=\{v'\}$ for some vertex $v'$ of $K$}. The main proof idea for this case is essentially borrowed from \cite[Lemma~3.1 (Case~2)]{Bianchi02}. Let us first show that $S(K,v)$ or $-S(K,v)$ is a supporting cone of $L'$ for some $L'\in[L]$. We do this by proving an equivalent statement, namely that $N(K,v)$ or $-N(K,v)$ is a normal cone of $L$ (with respect to an appropriate vertex of $L$), because the arguments for the normal cones will turn out to be more natural. We distinguish the following subcases:

\emph{Case 2.1: There exists $u'$ such that $F(K,-u')=\{v\}$ and $F(K,u')$ is an edge.}

\emph{Case 2.2: For each $u'$ such that $F(K,-u')=\{v\}$, the set $F(K,u')$ is a vertex of $K$}.

For Case 2.1 we see that $N(K,v)$ or $-N(K,v)$ is a normal cone of $L$ (this follows directly from Case~1). For Case 2.2 observe that $F(K,u')$ is the same vertex for each $u'$ with $F(K,-u')=\{v\}$. In fact, otherwise we could find two directions $u_1,u_2$ such that $F(K,-u_i)=\{v\}$ for $i\in\{1,2\}$ and $F(K,u_i)$ are consecutive vertices of $K$. But then the outer normal $u'$ of the edge joining $F(K,u_1)$ with $F(K,u_2)$ satisfies $F(K,-u')=\{v\}$, a contradiction to the assumption for Case~2.2. Thus, we have $\{v'\}=F(K,u')$ for each $u'$ with $F(K,-u')=\{v\}$. It follows that $-N(K,v) \subseteq N(K,v')$. If $-N(K,v)=N(K,v')$, we apply Lemma~\ref{lem:edges-up-to-flips} and infer that $-N(K,v)$ and $N(K,v')$ are normal cones of $L$, as well. Otherwise $-N(K,v) \varsubsetneq N(K,v')$, and thus there exists a direction $u'$ such that $F(K,-u')$ is an edge containing $v$ and $F(K,u')=\{v'\}$. From Case~2.1, it follows that $N(K,v')$ is determined by $g_K$, up to a sign. That is, for some $L' \in [L]$ the point $v'$ is a vertex of $L'$ and $N(K,v')=N(L',v')$. Then Lemma~\ref{lem:edges-up-to-flips} implies that the set of those edge normals of $K$  and $L'$ that lie in $-N(K,v')$ coincide. Consequently $N(K,v)=N(L',v)$. 

Summarizing the arguments of the subcases, we conclude that $N(K,v)$ or $-N(K,v)$ is a normal cone of $L$. The arguments of the subcases are still valid if we replace $u'$ by $-u'$ and exchange $v$ and $v'$. Thus, we also conclude that $N(K,v')$ or $-N(K,v')$ is a normal cone of $L$. The latter implies that \emph{either} $N(K,v)$ and $N(K,v')$ \emph{or} $-N(K,v)$ and $-N(K,v')$ are both normal cones of $L$ (e.g., $N(K,v)$ and $-N(K,v')$ cannot be simultaneously normal cones of $L$ since their interiors intersect). Taking into account that $F(K,u) \cup F(K,-u)$ and $F(L,u) \cup F(L,-u)$ are translations of each other (see Lemma~\ref{lem:edges-up-to-flips}), we arrive at the assertion of Lemma~\ref{lem:all-cones} within Case~2.
\end{proof}

\begin{remark}
 Lemma~\ref{lem:edges-up-to-flips} and Lemma~\ref{lem:all-cones} already imply that if $K, L \in \LCONV{2}$ satisfy $g_K=g_L$ and if $K$ (and thus $L$) does not contain parallel edges of different lengths, then $m'(K)\ge \delta(K)^2+\delta(K)+1$ ensures that $[K]=[L]$. This is a special case of Theorem~\ref{thm:main-thm}.
\end{remark}

If $K$ contains parallel edges of different lengths, the information gained so far by Lemma~\ref{lem:edges-up-to-flips} and Lemma~\ref{lem:all-cones} is in general not enough to guarantee that $K$ can be reassembled from $g_K$ up to translations and reflections. But $g_K$ provides more information. In fact, Lemma~\ref{lem:Bianchi-main} (which is essentially the same as \cite[Lemma 4.1]{Bianchi02}) tells us that if we could \emph{not} reconstruct $[K]$ from $g_K$ in a unique fashion, then for another solution $L$ there must occur certain `symmetries' in the parts where the boundaries of $K$ and $L$ (or translations or reflections of $L$) overlap. This will be the key argument in the proof of Theorem~\ref{thm:main-thm}.

By $\mathbb{S}^1$ we denote the unit circle $\setcond{u\in\real^2}{ \sprod{u}{u}=1 }$ in $\real^2$.

\begin{lemma}\label{lem:Bianchi-main}
 Let $K, L\in \LCONV{2}$ be such that $g_K=g_{L}$ and $m(K)\geq \delta(K)^2+\delta(K)+1$. Assume that there is a nonempty arc $U\subseteq \mathbb{S}^1$ that satisfies
 \begin{align}
 \label{eq:arc}
 G(U):=&\bigcup_{u\in U} F(K,u) = \bigcup_{u\in U} F(L,u), &
 G(-U):=&\bigcup_{u\in -U} F(K,u) = \bigcup_{u\in -U} F(L,u)
 \end{align}
 and that is maximal (with respect to inclusion) with the property~\eqref{eq:arc}. Let $A$ be the inclusion-maximal arc satisfying $G(U) \subseteq A \subseteq \bd \overline{K} \cap \bd \overline{L}$ and $B$ be the inclusion-maximal arc satisfying $G(-U) \subseteq B \subseteq \bd \overline{K} \cap \bd \overline{L}$. Let $[K]\neq [L]$. Then one of the following holds: $|A|=1$ or $|B|=1$, or $A$ and $B$ are parallel line segments, or $A$ is a reflection of $B$.
\end{lemma}

\begin{proof} Put $m:=m(K)$ and $\delta:=\delta(K)$. We follow the lines of the proof of~\cite[Lemma 4.1]{Bianchi02}, but replace a `continuous' argument that makes use of derivatives with a discrete one. First, observe that $[K]\neq [L]$ implies that neither $A$ nor $B$ coincides with $\bd\overline{K}$. Observe that $U$ is closed.

 If $|A|=1$ or $|B|=1$ we are done. So assume that both $A$ and $B$ contain more than one point. Let $a_1$, $a_2$ be the endpoints of $A$, and let $b_1$, $b_2$ be the endpoints of $B$. Choose the labeling in such a way that $a_1$, $a_2$, $b_1$, and $b_2$ are in counterclockwise order on $\bd\overline{K}$. Further on, let $u_1$, $u_2$ be the first and second endpoint of $U$ in counterclockwise order, respectively. We claim that
 \begin{align}
 \label{eq:u1}
 u_1&\in N(K,a_1)\cap N(L,a_1), & -u_1&\in N(K,b_1)\cap N(L,b_1),\\
 \label{eq:u2}
 u_2&\in N(K,a_2)\cap N(L,a_2), & -u_2&\in N(K,b_2)\cap N(L,b_2).
 \end{align}
Thus, we need to show $u_i \in N(K,a_i)$, $u_i \in N(L,a_i)$, $-u_i \in N(K,b_i)$, $-u_i \in N(L,b_i)$ for $i \in \{1,2\}$. We prove $u_1\in N(K,a_1)$ by contradiction; the remaining assertions can be settled in complete analogy. By construction we have $F(K,u_1)\cap A\neq \emptyset$. Let $U'$ be the set of vectors $u'\in\mathbb{S}^1$ which, in counterclockwise order, strictly precede $u_1$ and strictly follow every $v\in N(K,a_1)$. If $a_1\notin F(K,u_1)$, then $U'$ is nonempty. For $u'\in U'\cup\{u_1\}$ we have $u'\notin N(K,a_1)$ and for $u'\in U'$ the set $F(K,u')$ is contained in the relative interior of $A$. By definition of $A$ and due to \eqref{eq:edges-up-to-flips} we have $F(K,-u')=F(L,-u')$ for each $u' \in U' \cup \{u_1\}$. This contradicts the maximality of $U$ and concludes the proof of the claim.

 Now, for $i\in\{1,2\}$, let $c_i=(a_i+b_i)/2$ and let $A'_i$ be the maximal subarc of $A$ that starts in $a_i$ and in which each point is also present in $B$ after reflection in $c_i$. It is possible that $A'_i$ contains just one point, in which case we say that $A'_i$ is \emph{degenerate}. Further on, let $B'_i$ be the reflection of $A'_i$ in $c_i$. By construction we have $a_i\in A'_i$ and $b_i\in B'_i$. We also observe that $a_i$ and $b_i$ are vertices of $K$ or $L$ for $i\in\{1,2\}$, but if parallel edges of different lengths exist in $K$ and $L$, they do not have to be vertices of both $K$ and $L$.

 \emph{Case 1: For each $i\in\{1,2\}$ we have $A'_i=A$ or $B'_i=B$.} If for some $i\in\{1,2\}$ we have $A'_i=A$ and $B'_i=B$, then $A$ is a reflection of $B$ and we are done. So assume that for each $i\in\{1,2\}$ we \emph{either} have $A'_i=A$ \emph{or} $B'_i=B$. Then, in view of the symmetries arising from the exchanging of $A$ and $B$ and/or replacing $i$ by $3-i$, it is sufficient to distinguish the following subcases:

 \emph{Case 1.1: $A'_1 = A \ne A'_2$ and $B'_1 \ne B = B'_2$.}

 \emph{Case 1.2: $A'_1 = A = A'_2$ and $B'_1 \ne B \ne B'_2$.}

 In Case 1.1 we have
 \begin{equation*}
 A=A'_1 = 2 c_1 - B'_1 \varsubsetneq 2 c_1 - B = 2 c_1 - B'_2 = 2 c_1 - 2 c_2 + A'_2 \varsubsetneq 2 c_1 - 2 c_2 + A.
\end{equation*}
But this implies that $A$ is properly contained in a translation of $A$, which yields a contradiction. In Case~1.2, $B'_1$ and $B'_2$ are two distinct translations of each other contained in $B$ and each sharing an endpoint with $B$. This is only possible if $B'_1, B'_2$ and $B$ are parallel line segments.

\emph{Case 2: There exists $i\in\{1,2\}$ such that $A'_i\neq A$ and $B'_i\neq B$.} We prove that this case cannot occur by showing that it implies $g_K\neq g_{L}$, a contradiction. We show the arguments for $i=1$; the case $i=2$ can be proved in complete analogy. We refer to Figure~\ref{fig:Bianchi-main} for an illustration of the following arguments.

\begin{figure}
 \centering
 \tikzhide{\tikzBianchiMain}%
 \pdfhide{\includegraphics{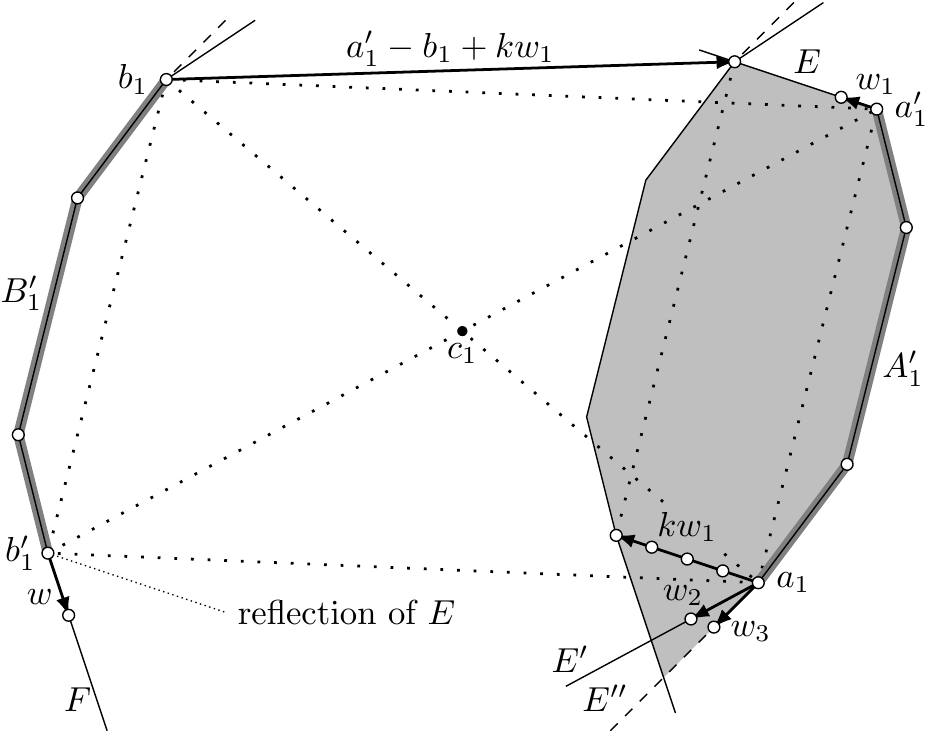}}
 \parbox[t]{0.9\textwidth}{%
 \caption{Illustration of the proof of Lemma~\ref{lem:Bianchi-main}, Case 2. The relevant parts of the boundaries of $\overline{K}$ and $\overline{L}$ are drawn with solid and dashed lines, respectively. In the figure, the arcs $A_1'$ and $B_1'$ are nondegenerate; collapsing these arcs to points, one obtains an illustration for the degenerate case.}
 \label{fig:Bianchi-main}}
 \end{figure}%

Let $a'_1$ and $b'_1$ be chosen so that $\{a_1,a'_1\}$ is the set of endpoints of $A'_1$ and $\{b_1,b'_1\}$ is the set of endpoints of $B'_1$. (If $A'_1$ and thus $B'_1$ is degenerate, then $a_1=a'_1$ and $b_1=b'_1$.) The point $a'_1$ or $b'_1$ is a vertex of both $K$ and $L$, but not necessarily both of them. (Again, the latter can happen if $K$ and $L$ have parallel edges of different lengths.) But there is an edge $E$ of $\overline{K}$ or $\overline{L}$ that contains $a'_1$, that is not completely contained in $A'_1$, and that is a subset of both $\bd\overline{K}$ and $\bd\overline{L}$. Similarly, there is an edge $F$ of $\overline{K}$ or $\overline{L}$ that contains $b'_1$, that is not completely contained in $B'_1$, and that is a subset of both $\bd\overline{K}$ and $\bd\overline{L}$. By the choice of $A'_1$, the edges $E$ and $F$ are not parallel. Observe that, by the definition of $m$, the sets $E\setminus A'_1$ and $F\setminus B'_1$ contain at least $m-1$ lattice points.

By definition of $U$ there exist nonparallel edges $E'$ and $E''$ of $K$ and $L$, respectively, that contain $a_1$ (not necessarily as a vertex) and are not completely contained in $A'_1$. By the definition of $m$, the sets $E'\setminus A'_1$ and $E'\setminus A'_1$ contain at least $m-1$ lattice points.

Choose $w_1,w_2,w_3,w\in\integer^2$ such that $\gcd(w_i)=1$ for each $i\in\{1,\ldots,4\}$ and such the following conditions are fulfilled:
\begin{itemize}
\item $w_1$, $w_2$, $w_3$, $w$ are parallel to $E$, $E'$, $E''$,  $F$,
 respectively,
\item $a'_1+w_1\in \bd \overline{K}\setminus A_1'$, $a_1+w_2\in \bd \overline{K}\setminus A_1'$, $a_1+w_3\in \bd \overline{L}\setminus A_1'$, $b'_1+w\in \bd \overline{K}\setminus B_1'$.
\end{itemize}
	After possibly exchanging $U$ with $-U$ and/or interchanging the roles of $K$ and $L$, the following inequalities are ensured:\footnote{Recall that, for $x,y\in\real^2$, one has $\det(x,y)>0$ if and only if $(x,y)$ is a positively oriented basis of $\real^2$.}
	\begin{align}
		\det(w_1,w) & > 0, \label{eq:det:w_1:2} \\
		\det(w_1,w_i) & > 0 \quad (i \in \{2,3\}),  \label{eq:det:w_1:w_i} \\
		\det(w_i,w) & > 0 \quad (i \in \{2,3\}), \label{eq:det:w_i:w} \\
		\det(w_2,w_3) & > 0. \label{eq:det:w_2:w_3}
	\end{align}

	This is shown as follows.  By assumption for Case~2 we have $\det(w_1,w) \ne 0$. Thus, possibly interchanging the roles of $U$ and $-U$ (and by this also the roles of $A$ and $B$) we achieve $\det(w_1,w)>0$. By construction $\det(w_1,w_i) \ge 0$ for each $i \in \{1,2\}$, and these inequalities are even strict. In fact, if we assume that, say, for $i=2$ one has $\det(w_1,w_i) = 0$, then $w_1=w_2$ and $\overline{K}= \overline{A_1' \cup B_1'}$. Hence $\overline{K}$ is centrally symmetric and using Lemma~\ref{lem:edges-up-to-flips} we deduce $K=L$, a contradiction.  Let us show the nonstrict inequalities $\det(w_i,w) \ge 0$ for $i \in \{1,2\}$. If $A_1'$ is nondegenerate, the inequalities follow directly by construction (taking into account that $a_1,a_1',b_1,b_1'$ lie in counterclockwise order on $\bd \overline{K}$). If $A_1'$ is degenerate, then $b_1 \in F$.  Hence, using $a_1 \in E'\cap E''$, the relations in~\eqref{eq:u1}, and the definition of the vectors $w, w_2, w_3$, we see that $\sprod{u_1}{w_i} \le 0$ for each $i \in \{2,3\}$ and $\sprod{u_1}{w} \ge 0$. Taking into account the fact that $a_1,a_2,b_1,b_2$ are in counterclockwise order on $\bd \overline{K}$ we obtain the inequalities $\det(w_i,w) \ge 0$ for $i \in \{2,3\}$. The inequalities are strict. In fact, if we assume that, say, for $i=2$ one has $\det(w_2,w)=0$, then $w$ is parallel to $a_1-b_1'$. The latter implies $\det(w_1,w) \le 0$, a contradiction to $\det(w_1,w) > 0$. Possibly interchanging the roles of $K$ and $L$ we can assume that $\det(w_2,w_3)>0$, which can also be stated as $w_3\notin S(K,a_1)$. 

	By Lemma~\ref{lem:k-exists} there exists $k\in\natur$ with $k\leq m-1$ and there exist real numbers $ \alpha_2,\alpha_3,\beta_2,\beta_3$ such that $k w_1 = \alpha_i w_i - \beta_i w$ and $1\leq |\alpha_1| \leq m-1$ and $1\leq |\beta_i| \leq m-1$ for each $i\in\{2,3\}$. By Cramer's rule and \eqref{eq:det:w_1:2}--\eqref{eq:det:w_i:w}, we see that both $\alpha_i$ and $\beta_i$ are positive for each $i \in \{1,2\}$. By construction $a_1' + [o,k] w_1 \in \bd \overline{K} \cap \bd \overline{L}$. Due to convexity of $\overline{K}$ and $\overline{L}$ we also have 
\[a_1 + [0,k] w_1 \in \overline{\{b_1,b_1',a_1,a_1',a_1'+k w_1 \}} \subseteq \overline{K} \cap \overline{L}.\]
 Using the definition of $m$ we conclude that the triangle with vertices $a_1,a_1+k w_1,a_1 + \alpha_2 w_2$ and the triangle with vertices $a_1,a_1+k w_1, a_1 + \alpha_3 w_3$ are both contained in $\overline{K} \cap \overline{L}$. Now, using $\alpha_3 \ge 1$ and~\eqref{eq:det:w_2:w_3}, we have 
\[a_1+ w_3
\notin 
\overline{K}\cap(\overline{K}+a'_1-b_1+k w_1)
\varsubsetneq
\overline{L}\cap(\overline{L}+a'_1-b_1+k w_1)
\ni
a_1+ w_3,\]
giving $g_K(a'_1-b_1+k w_1)<g_L(a'_1-b_1+k w_1)$, a contradiction.
\end{proof}

\begin{remark}
 The argument of Lemma~\ref{lem:Bianchi-main} can be re-translated for the case of covariograms of convex polygons. This way one can establish a more discrete version of the proof of \cite[Lemma 4.1]{Bianchi02} and, by this, of the main result in \cite{Bianchi02}.
\end{remark}

 \begin{proof}[Proof of Theorem~\ref{thm:main-thm}]
This proof is essentially the same as the proof of \cite[Theorem 1.1]{Bianchi02}. We repeat it here as a service to the reader.

Let $K, L \in \LCONV{2}$ be sets with $g_K=g_L$ and put $m:=m(K)$ and $\delta:=\delta(K)$. It remains to show that if $m \ge \delta^2+\delta+1$, then $[K]=[L]$. To this end assume $[K]\neq [L]$. We will prove that $K$ and $L$ are centrally symmetric. This implies that $\overline{K}$ is a translation of $\frac{1}{2} \overline{DK}$; the same is true for $L$ since $\overline{DK}=\overline{DL}$. This yields the result that $K$ equals $L$ up to translations, a contradiction. 

To prove central symmetry of $K$ and $L$, let $v$ and $v'$ be opposite vertices of $K$, i.e.,
\[\intr N(K,v)\cap (-\intr N(K,v'))\neq \emptyset.\]
By Lemma~\ref{lem:all-cones} there exists $L'\in[L]$ such that $v$ and $v'$ are vertices of $L'$, $N(K,v)=N(L',v)$, and $N(K,v')=N(L',v')$. Now we apply Lemma~\ref{lem:Bianchi-main} with $U$ taken so that all the normalized vectors in $N(K,v)$ are in $U$. Let $A$ and $B$ be defined as in the statement of Lemma~\ref{lem:Bianchi-main}. Observe that $A$ and $B$ both are neither singletons nor line segments, because $\bd\overline{K}$ and $\bd\overline{L'}$ coincide in a neighborhood of both the vertices $v$ and $v'$. So Lemma~\ref{lem:Bianchi-main} implies that $A$ is a reflection of $B$.

We have shown that the normal cones of opposite edges of $K$ and $L$ are reflections of each other. Consequently, each edge of $K$ (resp.\ $L$) is parallel to another edge of $K$ (resp.\ $L$). It remains to show that all pairs of parallel edges of $K$ (and thus $L$) have equal length to prove that $K$ and $L$ are centrally symmetric. So let $E=\overline{\{v_1,v_2\}}$ and $F=\overline{\{v_1',v_2'\}}$ be two parallel edges of $K$, where $v_1$, $v_2$, $v_1'$, and $v_2'$ are in counterclockwise order of $\bd\overline{K}$. By Lemma~\ref{lem:edges-up-to-flips} we can choose $L'\in [L]$ such that $E$ and $F$ are also edges of $L'$. Now both pairs $v_1,v_1'$ and $v_2,v_2'$ are opposite vertices of $K$ as well as of $L'$. This gives
\begin{align*}
 N(K,v_1)&=N(L',v_1)=-N(K,v_1')=-N(L',v_1'),\\
 N(K,v_2)&=N(L',v_2)=-N(K,v_2')=-N(L',v_2').
\end{align*}
So the boundaries of $\overline{K}$ and $\overline{L}$ coincide also in a neighborhood of $E$ and $F$. Lemma~\ref{lem:Bianchi-main} then shows that $E$ must be a reflection of $F$, and so they have the same length.
 \end{proof}

\begin{remark} \label{rem:higher-dimensions}
	Extending Theorem~\ref{thm:main-thm} to higher dimensions seems to be a nontrivial task. Recent results for the continuous covariogram problem such as those presented in \cite{Bianchi09} or \cite{Bianchi-Gardner-Kiderlen-2011} suggest that an analogue of Theorem~\ref{thm:main-thm} for higher dimensions might exist.  However, at least for dimension $d \ge 4$, no straightforward generalization of Theorem~\ref{thm:main-thm} seems to be possible. In fact, for $d \ge 4$ one can borrow the construction of \cite[Theorem~1.2]{Bianchi05} to obtain a vast class of nontrivially homometric pairs of lattice convex sets: If $K \subseteq \integer^{\ell}$ and $L \subseteq \integer^m$ ($\ell, m \ge 2$) are both lattice-convex and not centrally symmetric, then the lattice convex sets $K \times L$ and $K \times (-L)$ in $\integer^{\ell + m}$ form a nontrivially homometric pair. So an essentially different approach seems to be needed to cope with higher dimensions.
\end{remark}

\section{Proof of Theorem~\ref{thm:homometry:from:lat:wid:one}}\label{sec:HomometricPairs}

For the remainder of this section we assume that the situation of Theorem~\ref{thm:homometry:from:lat:wid:one} is given, that is, $k$ and $\ell$ are integers with $k>\ell\geq 0$, and $T=T_1 \cup T_2$ with
\begin{align*}
 T_1 &:= \{0,\ldots,k\} \times \{0\}, & T_2 &:= \{0,\ldots,\ell\} \times \{1\}.
\end{align*}
Further we fix the vectors $w_1 := (-k-1,1)$, $w_2 := (\ell + 1, 1)$ and the lattice $\bL := \integer w_1 + \integer w_2$. 

\begin{lemma} \label{lattice:T:decomposition}
	With $\bL$ and $T$ defined as above one has $\integer^2 = \bL \oplus T$.
\end{lemma}
\begin{proof}
	It can be verified directly that the sum of $T$ and $\bL$ is direct if and only if $(DT) \cap \bL = \{o\}$. We have 
	\[ 
		DT = \{-k,\ldots,k\} \times \{0\} \cup \{-k,\ldots,\ell\} \times \{1\} \cup \{-\ell,\ldots,k\} \times \{-1\}.
	\]
	Thus, we see that $(DT) \cap \bL = \{o\}$ is in fact fulfilled. For showing $\integer^2 = \bL \oplus T$ we compare $\integer^2$ and $\bL \oplus T$ modulo $\bL$. We have 
	\begin{align*}
		\card{(T \oplus \bL) / \bL } = \card{T} = k+\ell+2 = |\det(w_1,w_2)| = \det \bL = \card{\integer^2 / \bL}.
	\end{align*}
	The above equality $\card{(T \oplus \bL) / \bL} = \card{T}$ holds since the sum of $T$ and $\bL$ is direct, and the equality $\det \bL = \card{\integer^2/ \bL}$ is a standard fact. Taking into account the inclusion $(T \oplus \bL) / \bL \subseteq \integer^2 / \bL$, we obtain $(T \oplus \bL) / \bL = \integer^2 / \bL$. The latter implies $T \oplus \bL = \integer^2$.
\end{proof}

In this section we give a proof of Theorem~\ref{thm:homometry:from:lat:wid:one}. We shall use several standard graph-theoretic notions (see, e.g., \cite{MR2159259} for a comprehensive account on graph theory). The graphs we use will all be undirected and in some cases infinite. As usual, $V(G)$ stands for the node set of a graph $G$. Under a \term{grid graph} we shall understand a Cartesian product of two graph-theoretic paths (where a graph-theoretic path is allowed to be finite, infinite in one direction, or infinite in two directions).
Given a set $S \subseteq \integer^2$ we define an undirected graph $G_S$ as follows. We declare two nodes $s,s' \in S$ to be adjacent if and only if $s-s' \in \{\pm w_1, \pm w_2\}$ for $k - \ell >1$ and if and only if $s-s' \in \{\pm w_1, \pm w_2, \pm (w_1+w_2)\}$ for $k- \ell =1$. Thus, for $k- \ell >1$, $G_{\bL}$ is an infinite grid graph and $G_S$ is the subgraph of $G_{\bL}$ induced by the node set $S$ (where $S \subseteq \bL$).

Let us first give a sketch of the proof of Theorem~\ref{thm:homometry:from:lat:wid:one}. The lengthy part of the proof is to provide enough auxiliary information for the implication \eqref{item:direct-and-convex} $\Rightarrow$ \eqref{item:S-is-special} of the theorem. This is done in a a number of lemmas. In a first step we prove that~\eqref{item:direct-and-convex} implies that the graph $G_S$ is connected (Lemma~\ref{lem:G_S-is-connected}). Having shown this, we derive that the lattice-convexity of $S \oplus T$ implies certain combinatorial conditions on $G_S$ (namely, the presence of certain pairs in $S$ force `neighboring' pairs to be also elements of $S$; see Lemma~\ref{lem:char-grid-graph} and Lemma~\ref{lem:char-quasi-grid}). These conditions together with the connectedness of $G_S$ imply~\eqref{item:S-is-special}. From this we can establish \eqref{item:direct-and-convex} $\Rightarrow$ \eqref{item:S-is-special}; the reverse implication will follow quite straightforwardly.

When dealing with a set $S \subseteq \integer^2$ such that the sum of $S$ and $T$ is direct and $K:= S \oplus T$ is lattice-convex we introduce the following notations. For $i\in\{1,2\}$ by $\cT_i$ we denote the set of all $s+T_i$ with $s \in S$. We set $\cT:=\cT_1 \cup \cT_2$. Then $K$ is a disjoint union of the elements of $\cT$. For $h \in \integer$, the set $I_h:=(\integer \times \{h\}) \cap K$ is the horizontal section of $K$ at the height $h$. The set $I_h$ is empty or a (probably infinite) interval of integer points. In the latter case it decomposes into the disjoint union of the elements of $\cT(h):= \setcond{J \in \cT}{J \subseteq I_h}$. We introduce a consecutive order on $\cT(h)$ by ordering elements from left to right. In particular, for distinct $J_1, J_2 \in \cT(h)$, $J_1$ strictly precedes $J_2$ if $J_1 \cup J_2$ is lattice-convex and the first coordinates of the points of $J_1$ are smaller than the first coordinates of the points of $J_2$.

\begin{lemma} \label{lem:alternate}
 Let $S \subseteq \integer^2$ be such that the sum of $S$ and $T$ is direct and lattice-convex. Then for each $h \in \integer$, the elements of $\cT_1$ and $\cT_2$ alternate in $\cT(h)$ with respect to the consecutive order of $\cT(h)$.
\end{lemma}
\begin{proof} We show by contradiction that neither two elements of $\cT_1$ nor two elements of $\cT_2$ can be consecutive in $\cT(h)$. Let $s, s' \in S$.

If the set $s + T_2$ strictly precedes $s' + T_2$ in $\cT(h)$, then $(s+T) \cap (s'+T) \ne \emptyset$, a contradiction to the fact that the sum of $S$ and $T$ is direct. Thus, for each $h \in \integer$ the set $\cT(h)$ does not contain two consecutive elements of $\cT_2$.

If $s + T_1$ strictly precedes $s'+T_1$, then there are precisely $k-\ell$ integer points between $s+T_2$ and $s' + T_2$ within $I_{h+1}$. Since no two elements of $\cT_2$ are consecutive in a horizontal section of $K$, we see that there must be an element of $\cT_1$ between $s+T_2$ and $s'+T_2$. Elements of $\cT_1$ have cardinality $k+1$, hence $k- \ell \ge k+1$, a contradiction to $\ell \ge 0$.
\end{proof}

\begin{lemma} \label{lem:consecutive-short-sections}
 Let $S \subseteq \integer^2$ be such that the sum of $S$ and $T$ is direct and lattice-convex. Let $h \in \integer$ be such that $\cT(h)$ consists of precisely one element $s+T_2$ with $s \in S$ and $\cT(h+1)$ consists of precisely one element $s'+T_1$ with $s' \in S$. Then $k-\ell=1$ and $s'-s = (-1,2) = w_1+w_2$.
\end{lemma}
\begin{proof}
Let $P$ be the convex hull of $(s + T_1) \cup (s'+T_1)$. Then $P$ is a parallelogram with horizontal edges of length $k$. Therefore, the section of $P$ by $\real \times \{h\}$ is also of length $k$; see Figure~\ref{fig:pairs}(a) for an illustration. This section then contains at least $k$ integer points. By lattice-convexity of $S \oplus T$, all those points belong to $S \oplus T$. By the assumption, $I_h$ contains precisely $\ell+1$ points. Hence $\ell + 1 \ge k$. Since we assume $k > \ell$, we deduce $k-\ell =1$.

Clearly, $s'-s=(t,2)$ for some $t \in \integer$. If $t \ge 0$, then $s+w_2$ belongs to the convex hull of $(s+T) \cup (s'+T)$, a contradiction to the fact that $s+T_2$ is the intersection of $K$ with $\integer \times \{h\}$. If $t < -1$, then $s+(-1,1)$ belongs to the convex hull of $(s+T) \cup (s'+T)$, a contradiction the fact that $s+T_2$ is the intersection of $K$ with $\integer \times \{h\}$. 
\end{proof}

\begin{lemma} \label{lem:segment:of:integers}
Let $S \subseteq \integer^2$ be nonempty and such that the sum of $S$ and $T$ is direct and $K:= S \oplus T$ is lattice-convex.  Then $\setcond{h \in \integer}{I_h \ne \emptyset}$ is an interval of of integers, i.e., a set of the form $\setcond{i\in\integer}{h_1\leq i\leq h_2}$ for some $h_1,h_2\in\integer\cup\{\pm \infty\}$ satisfying $h_1 \le h_2$.
\end{lemma}
\begin{proof} 
We consider arbitrary $h_1, h_2 \in \integer$ satisfying $I_{h_1} \ne \emptyset \ne I_{h_2}$. It suffices to show that $I_h \ne 0$ for every $h$ with $h_1 \le h \le h_2$. We may assume without loss of generality that $\cT(h_i)$ contains an element $J_i$ of $\cT_1$ for each $i\in\{1,2\}$; otherwise we decrease $h_1$ or $h_2$ or both by $1$. If $h_2\leq h_1$ we are done. Otherwise we consider the parallelogram $P:=\overline{J_1 \cup J_2}$, see Figure~\ref{fig:pairs}(b) for an illustration. For each $h \in \integer$ with $h_1 \le h \le h_2$ the set $(\real \times \{h\}) \cap P$ is a section of length $k$. Thus, $(\integer \times \{h\}) \cap P$ contains a least $k>0$ integer points. Since in view of convexity we have $(\integer \times \{h\}) \cap P \subseteq I_h$, we see that $I_h \ne \emptyset$.
\end{proof}

\begin{lemma} \label{lem:G_S-is-connected}
Let $S \subseteq \integer^2$ be nonempty and such that the sum of $S$ and $T$ is direct and lattice-convex. Then the graph $G_S$ is connected.
\end{lemma}
\begin{proof} If $\card{S}=1$ we are done. So let $s, s' \in S$ be distinct. We show that they are connected by a path in $G_S$. Choose $i, i' \in \{1,2\}$ and $h, h' \in \integer$ such that $s+T_i \in I_h$ and $s'+T_{i'} \in I_{h'}$. Without loss of generality let $h' \ge h$.

 \emph{Case 1: $h=h'$}. If $s+T_i$ and $s'+T_{i'}$ (in this order) are consecutive in $\cT(h)$, then for $(i,i')=(1,2)$ one has $s-s'=w_1$ and for $(i,i')=(2,1)$ one has $s'-s=w_2$. See Figure~\ref{fig:pairs}(c) for an illustration of the two cases. If $s+T_i$ and $s'+T_{i'}$ are not consecutive, then using Lemma~\ref{lem:alternate} and the previous observation for the consecutive elements of $\cT(h)$ lying between $s+T_i$ and $s'+T_{i'}$ we see that $s$ and $s'$ are connected in $S$.

 \emph{Case 2: $h'=h+1$}. If there exists $s'' \in S$ such that $s''+T_1 \in I_h$ and $s''+T_2 \in I_{h+1}$, then using Case 1 we see that $s$ and $s''$ are connected in $S$ and $s''$ and $s'$ are connected in $S$. But then also $s$ and $s'$ are connected in $S$. Otherwise, $\cT(h) \cap \cT_1 = \emptyset$ and $\cT(h+1) \cap \cT_2 = \emptyset$. In view of Lemma~\ref{lem:alternate}, we deduce that $\cT(h+1)$ consists of precisely one element of $\cT_1$ and $\cT(h)$ consists of precisely one element of $\cT_2$. Then, by Lemma~\ref{lem:consecutive-short-sections}, we see that $k-\ell=1$ and $s'-s=w_1+w_2$; so $s$ and $s'$ are connected in $S$.

 \emph{Case 3: $h'> h+1$}. By Lemma~\ref{lem:segment:of:integers}, $I_{h''} \ne \emptyset$ for each $h''$ that satisfies $h \le h'' \le h'$. Thus, we can employ the conclusion of Case 2 for those consecutive sections which are above $I_h$ and below $I_{h'}$.
\end{proof}

\begin{figure}
 \centering
 \tikzhide{\tikzPairs}%
 \pdfhide{\includegraphics{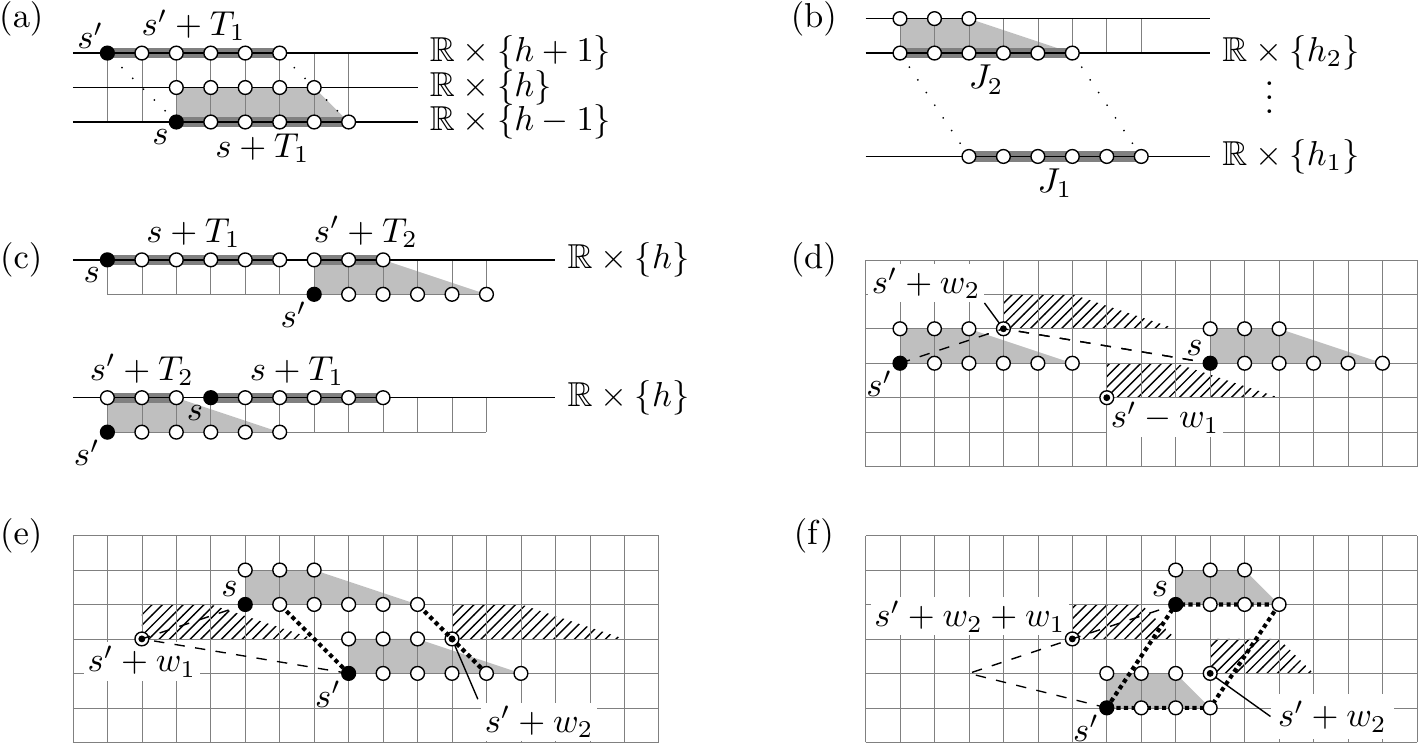}}
 \parbox[t]{0.9\textwidth}{%
 \caption{(a) Illustration of the proof of Lemma~\ref{lem:consecutive-short-sections}. (b) Illustration of the proof of Lemma~\ref{lem:segment:of:integers}. (c) Illustration of the proof of Lemma~\ref{lem:G_S-is-connected}, Case 1. The subcases $(i,i')=(1,2)$ and $(i,i')=(2,1)$ are depicted at the top and bottom, respectively. (d), (e), (f) Illustration of the proof of Lemma~\ref{lem:diagonal=>anotherdiagonal}~\ref{item:left-right=>down-up}, Lemma~\ref{lem:diagonal=>anotherdiagonal}~\ref{item:20.08.10,11:08}, Lemma~\ref{lem:diagonal=>anotherdiagonal}~\ref{item:20.08.10,11:09}, respectively.}
 \label{fig:pairs}}
\end{figure}%

\begin{lemma} \label{lem:diagonal=>anotherdiagonal}
Let $S \subseteq \integer^2$ be nonempty and such that the sum of $S$ and $T$ is lattice-convex. Let $s, s' \in S$. Then the following holds.
 \begin{enumerate}[I.]
 \item \label{item:left-right=>down-up} If $s-s' = w_2-w_1$, then $s'-w_1,s'+w_2 \in S$.
 \item \label{item:20.08.10,11:08} If $k-\ell>1$ and $s-s' = w_1+w_2$, then $s'+w_1, s'+w_2 \in S$ .
 \item \label{item:20.08.10,11:09} If $k - \ell = 1$ and $s-s'=w_1+ 2 w_2$, then $s'+w_2, s'+w_1+w_2 \in S$.
 \item \label{item:20.08.10,11:10} If $k - \ell=1$ and $s-s'= 2 w_1 + w_2$, then $s'+w_1, s'+w_1+w_2 \in S$.
 \end{enumerate}
\end{lemma}
\begin{proof} For an illustration of Parts~\ref{item:left-right=>down-up},~\ref{item:20.08.10,11:08}, and~\ref{item:20.08.10,11:09} and the according proof we refer to Figure~\ref{fig:pairs}(d), \ref{fig:pairs}(e), and \ref{fig:pairs}(f), respectively.

 \emph{Part~\ref{item:left-right=>down-up}}. After possibly translating $S$ we assume without loss of generality that $s, s' \in \real \times \{0\}$. Then $s+T_1$ and $s'+T_1$ are two elements of $\cT(0)$, and, in view of Lemma~\ref{lem:alternate}, there exists $s''+T_2$ with $s'' \in S$, lying between $s+T_1$ and $s'+T_2$. Since between $s+T_1$ and $s'+T_1$ there is space for only one copy of $T_2$, $s''+T_2$ and by this $s''$ is determined uniquely. It follows that $s''=s'-w_1$. The proof that $s'+w_2$ uses analogous arguments for sets in $\cT(1)$.

 \emph{Part~\ref{item:20.08.10,11:08}.} Without loss of generality let $s=o$, so $s'=(k-\ell,-2)$. Let us show that $(k-\ell-1,-1) \in S \oplus T$. We have $(k-\ell-1,-1) = \frac{1}{2} (k- \ell -2,0 ) + \frac{1}{2} s'$, where $(k-\ell-2,0) \in s + T_1 \subseteq s +T$ and $s'\in s' + T$. Thus, by the lattice-convexity of $S \oplus T$, we have $(k-\ell-1,-1) \in S\oplus T$. The lattice point $(k-\ell-1,-1)$ strictly precedes $s' + T_2$ in $I_{-1}$. Consequently, in view of Lemma~\ref{lem:alternate}, we see that $s-w_2=s'+w_1 \in S$.

 We show that $(k+1,-1) \in S \oplus T$. We have $(k+1,-1) = \frac{1}{2} (k,0) + \frac{1}{2} (k+2,-2)$, where $(k,0) \in s+T_1 \subseteq s+T$ and $(k+2,-2) \in s' + T_1$ (since $k > \ell +1$). Thus, by lattice-convexity of $S \oplus T$, $(k+1,-1) \in S \oplus T$. The point $(k+1,-1)$ directly follows $s' + T_2$ in $I_{-1}$. Thus, in view of Lemma~\ref{lem:alternate}, we see that $s-w_1=s'+w_2 \in S$.

 \emph{Part~\ref{item:20.08.10,11:09}.} Without loss of generality let $s'=o$, so $s=(k-1,3)$. Let $P$ be the convex hull of ${(s+T_1) \cup (s'+T_1)}$. The point $s'':=w_2=(k,1)=s'+w_2$ lies in $P$, and by this, belongs to $S \oplus T$. Since the point $s''$ strictly follows $s'+T_2$, in view of Lemma~\ref{lem:alternate} applied for elements of $\cT(1)$, we deduce $s''+T_2 \in \cT(1)$. Now, analogously, we use the point $(k-1,2)$, which strictly precedes $s''+T_2$ and lies in $P$ to show that $s'+w_1+w_2$ belongs to $S$. 

 \emph{Part~\ref{item:20.08.10,11:10}.} The proof is analogous to the proof of Part~\ref{item:20.08.10,11:09}.
\end{proof}

The assertions of Lemma~\ref{lem:diagonal=>anotherdiagonal} together with the connectedness of $G_S$ presented in Lemma~\ref{lem:G_S-is-connected} impose strong restrictions on $S$. This is the topic of Lemma~\ref{lem:char-grid-graph} and Lemma~\ref{lem:char-quasi-grid}.

\begin{lemma} \label{lem:char-grid-graph}
 Let $k - \ell > 1$ (and so $G_\bL$ is an infinite grid graph). Let $S \subseteq \bL$ be such that $G_S$ is connected and for all $s, s' \in S$ one has:
 \begin{alignat}{3} 
 s-s' & = w_2 - w_1 & & \quad\Rightarrow\quad & s'-w_1, s'+w_2
 \in S, \label{eq:2010.08.23,12:56} \\
 s-s' & = w_2 + w_1 & & \quad\Rightarrow\quad & s'+w_1, s'+w_2 \in
 S. \label{eq:2010.08.23,12:57}
 \end{alignat}
 Then $G_S$ is a grid graph, i.e.,
\begin{equation*}
 S = \setcond{i w_1 + j w_2}{i,j \in \integer
 \,\land\,
 \alpha_1 \le i \le \alpha_2
 \,\land\,
 \beta_1 \le j \le \beta_2}
\end{equation*}
 for some $\alpha_1, \alpha_2, \beta_1, \beta_2 \in \integer \cup \{-\infty,\infty\}$.
\end{lemma}

\begin{proof}
 In what follows we are only interested in the case that $S$ is finite. We give a proof for this, while the proof for the general situation is essentially the same. The following arguments are illustrated in Figure~\ref{fig:ShapeOfS}(a).
Let $G'$ be an inclusion-maximal grid graph contained in $G_S$. If $G' \ne G_S$, in view of connectedness of $G_S$, we can find two nodes $s,s' \in \integer^2$ adjacent in $G_S$ such that $s \in V(G')$ and $s' \in S \setminus V(G')$. Possibly translating or reflecting $S$, we assume that $V(G')= \{0,\ldots,\alpha\} w_1 + \{0,\ldots,\beta\} w_2$ for integers $\alpha, \beta \ge 0$ and $s'=i w_1 - w_2$ for some $i \in \{0,\ldots,\alpha\}$ or $s'=j w_2 - w_1$ for some $j\in \{1,\ldots,\beta\}$. We consider the case $s'=iw_1-w_2$, for the other case is handled analogously. Consecutively using \eqref{eq:2010.08.23,12:56} and \eqref{eq:2010.08.23,12:57}, we get $i' w_1 -w_2 \in V(G')$ for each $i' \in \{0,\ldots,\alpha\}$, a contradiction to the maximality of $G'$.
\end{proof}

\begin{figure}
 \centering \tikzhide{\tikzShapeOfS}%
 \pdfhide{\includegraphics{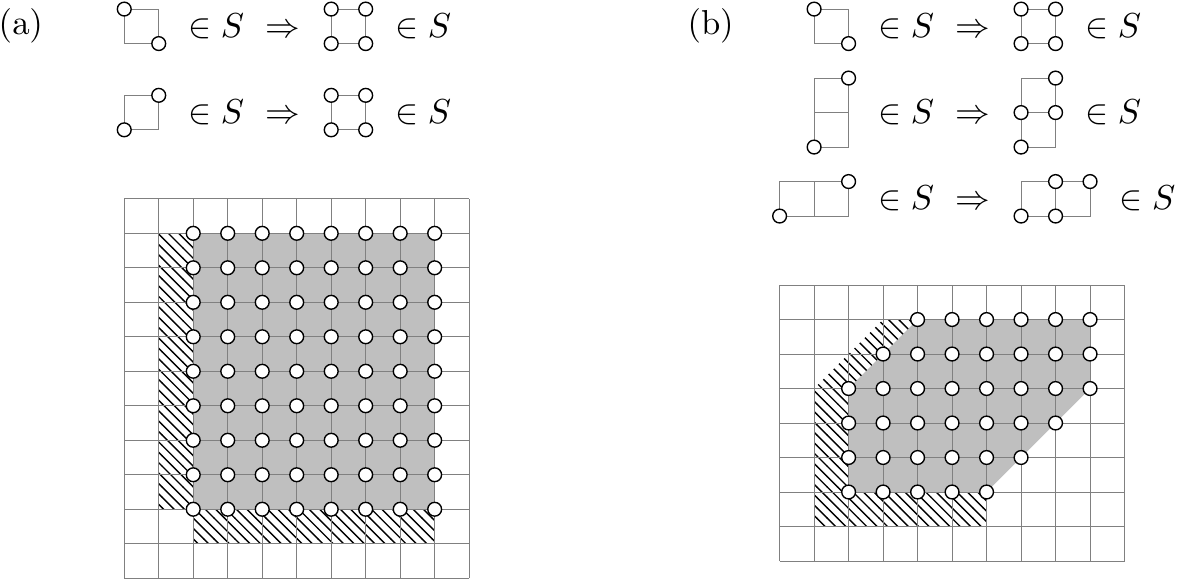}}
 \parbox[t]{0.9\textwidth}{%
   \caption{(a) and (b) illustrate the proofs of Lemma~\ref{lem:char-grid-graph} and  Lemma~\ref{lem:char-quasi-grid}, respectively. Due to affine invariance we exchange $\integer^2$ (and thus $\bL$) by another lattice in which one has $w_1=(1,0)$, $w_2=(0,1)$; this makes the figures invariant under the choice of $k$ and $\ell$ and thus easier to read. At the top of (a) and (b) the implications~\eqref{eq:2010.08.23,12:56},~\eqref{eq:2010.08.23,12:57} and~\eqref{eq:char-quasi-grid:a},~\eqref{eq:char-quasi-grid:b},~\eqref{eq:char-quasi-grid:c} are written in a pictographic style, respectively. At the bottom, the lattice points in the gray area indicate nodes of $G'$; if a lattice point on the boundary of the hatched area was contained in $G_S$, then $G'$ could be extended and would not be maximal. This also holds if $\overline{V(G')}$ is not full-dimensional and/or is unbounded and/or, in case of (b), has less than six edges.}
 \label{fig:ShapeOfS}}
\end{figure}%

\begin{lemma} \label{lem:char-quasi-grid}
 Let $k - \ell =1$. Let $S \subseteq \bL$ be such that $G_S$ is connected and for all $s, s' \in S$ one has:
 \begin{alignat}{3}
 \label{eq:char-quasi-grid:a}
 s-s' & = w_2 -w_1 & \quad\Rightarrow\quad & s'-w_1, s'+w_2 \in S, \\
 \label{eq:char-quasi-grid:b}
 s-s' & =w_1 + 2 w_2 & \quad\Rightarrow\quad & s'+w_2, s'+w_1 + w_2 \in S,\\
 \label{eq:char-quasi-grid:c}
 s-s' & =2 w_1 + w_2 & \quad\Rightarrow\quad & s'+w_1, s'+w_1+w_2 \in S.
 \end{alignat}
 Then
 \begin{equation} \label{special:hexagon} S = \setcond{i w_1 + j w_2}{i,j \in
 \integer \,\land\, \alpha_1 \le i \le \alpha_2 \,\land\, \beta_1 \le j \le
 \beta_2 \,\land\, \gamma_1 \le i-j \le \gamma_2}
 \end{equation}
 for some $\alpha_1,\alpha_2,\beta_1,\beta_2,\gamma_1,\gamma_2 \in \integer \cup \{-\infty,+\infty\}$.
\end{lemma}
\begin{proof}  The proof of Lemma~\ref{lem:char-quasi-grid} follows along the same lines as the proof of Lemma~\ref{lem:char-grid-graph}, we give a sketch. Again one considers be the inclusion-maximal subgraph $G'$ of $G$ whose node set can be given as in \eqref{special:hexagon}. If $G' \ne G_S$, then by Lemma~\ref{lem:G_S-is-connected} there exist $s \in V(G')$ and $s'\in S \setminus V(G')$ such that $s$ and $s'$ are adjacent in $G_S$. After possibly translating or reflecting $S$, we may assume that $s-s'\in\{w_1,w_2,w_1+w_2\}$. Successively applying~\eqref{eq:char-quasi-grid:a},~\eqref{eq:char-quasi-grid:b}, and~\eqref{eq:char-quasi-grid:c} one concludes that $G'$ was not maximal, a contradiction. Here one has to perform more case distinctions than in the proof of Lemma~\ref{lem:char-grid-graph}. See Figure~\ref{fig:ShapeOfS}(b) for an illustration.
\end{proof}

\begin{proof}[Proof of Theorem~\ref{thm:homometry:from:lat:wid:one}]
 \emph{\eqref{item:direct-and-convex} $\Rightarrow$ \eqref{item:S-is-special}.} If the sum of $S$ and $T$ is direct and $S \oplus T$ is lattice-convex, then Lemma~\ref{lem:G_S-is-connected} implies that $G_S$ is connected. Furthermore, in view Lemma~\ref{lem:diagonal=>anotherdiagonal}, we can apply Lemma~\ref{lem:char-grid-graph} or Lemma~\ref{lem:char-quasi-grid}, depending on whether $k-\ell > 1$ or $k-\ell =1$. This yields the necessity.

 \emph{\eqref{item:S-is-special} $\Rightarrow$ \eqref{item:direct-and-convex}}. Assume that $S$ is given as in \eqref{item:S-is-special}. Since $S \subseteq \bL$ and the sum of $\bL$ and $T$ is direct (see Lemma~\ref{lattice:T:decomposition}), we see that the sum of $S$ and $T$ is direct. It remains to prove that $S \oplus T$ is lattice-convex. For this we show that $\integer^2 \cap \overline{S \oplus T} \subseteq S \oplus T$. Taking into account $\overline{S+T}=\overline{S}+\overline{T}$ and using Lemma~\ref{lattice:T:decomposition} we obtain
 \begin{equation}
 \integer^2 \cap \overline{S \oplus T} = \integer^2 \cap (\overline{S} + \overline{T}) = (\bL + T) \cap (\overline{S} + \overline{T}). \label{eq:2010.08.23,15:18}
 \end{equation}

 \emph{Case~1: $k-\ell=1$.} We represent $\overline{S}$ by
 \begin{equation} \label{eq:convS=} \overline{S} = \setcond{x w_1 + y w_2}{x, y \in
 \real \,\land\, \alpha_1 \le x \le \alpha_2 \,\land\, \beta_1 \le y \le
 \beta_2 \,\land\, \gamma_1 \le x-y \le \gamma_2},
 \end{equation}
 where $\alpha_1,\alpha_2,\beta_1,\beta_2,\gamma_1,\gamma_2 \in \integer$.

 If we represent $p \in \overline{DT}$ as $p = x_p w_1 + y_p w_2$, then it can be verified directly that
 \begin{equation}
 -1 < x_p,y_p,x_p-y_p < 1. \label{eq:2010.08.23,15:21}
 \end{equation}

 In view of \eqref{eq:2010.08.23,15:18}, each element of $\integer^2 \cap \overline{S \oplus T}$ can be given as $s+ t = s'+ t'$ with $s \in \bL$, $t \in T$, $s' \in \overline{S}$, and $t' \in \overline{T}$. Let $s = i w_1 + jw_2$ for $i,j \in \integer$. We show that $s \in S$ arguing by contradiction. Assume that $s \not\in S$. Then, since $S = \bL \cap \overline{S}$, one of the inequalities on the right hand side of \eqref{eq:convS=} is violated for $x=i$ and $y=j$. But then, in view of \eqref{eq:2010.08.23,15:21}, the inequality which is violated for $x=i$ and $y=j$ is also violated for $x=i+x_p$ and $j+y_p$ with $p=t-t'$. This contradicts the fact that $s + t-t' = s' \in \overline{S}$. It follows that $s \in S$. Consequently, $\integer^2 \cap \overline{S \oplus T} \subseteq S + T$.

 \emph{Case~2: $k-\ell>1$.} The proof of this case follows along the same lines. For $x_p, y_p$ as for the previous case we have $-1<x_p,y_p<1$; the remaining arguments are analogous.
\end{proof}

\paragraph{Acknowledgment.} \emph{We thank the anonymous referee for valuable suggestions regarding a previous version of this paper.}


\begin{thebibliography}{DGN05}

\bibitem[AB07a]{AveBia07}
G.~Averkov and G.~Bianchi, \emph{Retrieving convex bodies from restricted
  covariogram functions}, Adv. in Appl. Probab. \textbf{39} (2007), no.~3,
  613--629.

\bibitem[AB07b]{Averkov-Bianchi-ArXiv-2007}
\bysame, \emph{Retrieving convex bodies from restricted covariogram functions},
  21\,pp., preprint, available at http://arxiv.org/abs/math/0702892, 2007.

\bibitem[AB09]{Averkov-Bianchi-2009}
\bysame, \emph{Confirmation of {M}atheron's conjecture on the covariogram of a
  planar convex body}, J. Eur. Math. Soc. \textbf{11} (2009), no.~6,
  1187--1202.

\bibitem[Ave09]{Averkov-detecting-cen-sym-2009}
G.~Averkov, \emph{Detecting and reconstructing centrally symmetric sets from
  the autocorrelation: two discrete cases}, Appl. Math. Lett. \textbf{22}
  (2009), no.~9, 1476--1478.

\bibitem[BBD10]{Benassi-Bianchi-DErcole-2010}
A.~Benassi, G.~Bianchi, and G.~D'Ercole, \emph{Covariogram of non-convex sets},
  Mathematika \textbf{56} (2010), 267--284.

\bibitem[BG07]{BaakeGrimm}
M.~Baake and U.~Grimm, \emph{Homometric model sets and window covariograms},
  Zeitschrift f{\"u}r Kristallographie \textbf{222} (2007), 54--58.

\bibitem[BGK11]{Bianchi-Gardner-Kiderlen-2011}
G.~Bianchi, R.~J. Gardner, and M.~Kiderlen, \emph{Phase retrieval for
  characteristic functions of convex bodies and reconstruction from
  covariograms}, J. Amer. Math. Soc. \textbf{24} (2011), no.~2, 293--343.

\bibitem[Bia02]{Bianchi02}
G.~Bianchi, \emph{Determining convex polygons from their covariograms}, Adv. in
  Appl. Probab. \textbf{34} (2002), no.~2, 261--266.

\bibitem[Bia05]{Bianchi05}
\bysame, \emph{Matheron's conjecture for the covariogram problem}, J. London
  Math. Soc. (2) \textbf{71} (2005), no.~1, 203--220.

\bibitem[Bia09a]{Bianchi09}
\bysame, \emph{The covariogram determines three-dimensional convex polytopes},
  Adv. Math. \textbf{220} (2009), no.~6, 1771--1808.

\bibitem[Bia09b]{Bianchi-cross-cov-2009}
\bysame, \emph{The cross covariogram of a pair of polygons determines both
  polygons, with a few exceptions}, Adv. in Appl. Math. \textbf{42} (2009),
  no.~4, 519--544.

\bibitem[Bia09c]{Bianchi-tom-cones-2009}
\bysame, \emph{Geometric tomography of convex cones}, Discrete Comput. Geom.
  \textbf{41} (2009), no.~1, 61--76.

\bibitem[BSV02]{MR1938112}
G.~Bianchi, F.~Segala, and A.~Vol{\v{c}}i{\v{c}}, \emph{The solution of the
  covariogram problem for plane {$\mathcal{C}^2_+$} convex bodies}, J.
  Differential Geom. \textbf{60} (2002), no.~2, 177--198.

\bibitem[DGN05]{Daurat-Gerard-Nivat-2005}
A.~Daurat, Y.~G{\'e}rard, and M.~Nivat, \emph{Some necessary clarifications
  about the chords' problem and the partial digest problem}, Theoret. Comput.
  Sci. \textbf{347} (2005), no.~1-2, 432--436.

\bibitem[Die05]{MR2159259}
R.~Diestel, \emph{Graph {T}heory}, third ed., Graduate Texts in Mathematics,
  vol. 173, Springer-Verlag, Berlin, 2005.

\bibitem[Gar06]{Gardner-book-2006}
R.~J. Gardner, \emph{Geometric {T}omography}, second ed., Encyclopedia of
  Mathematics and its Applications, vol.~58, Cambridge University Press,
  Cambridge, 2006.

\bibitem[GG97]{MR1376547}
R.~J. Gardner and P.~Gritzmann, \emph{Discrete tomography: determination of
  finite sets by {X}-rays}, Trans. Amer. Math. Soc. \textbf{349} (1997), no.~6,
  2271--2295.

\bibitem[GGZ05]{MR2160045}
R.J. Gardner, P.~Gronchi, and Ch. Zong, \emph{Sums, projections, and sections
  of lattice sets, and the discrete covariogram}, Discrete Comput. Geom.
  \textbf{34} (2005), no.~3, 391--409.

\bibitem[HK99]{MR1722457}
G.~T. Herman and A.~Kuba (eds.), \emph{Discrete {T}omography}, Applied and
  Numerical Harmonic Analysis, Birkh\"auser Boston Inc., Boston, MA, 1999,
  Foundations, algorithms, and applications.

\bibitem[HK07]{HerK07}
G.~T. Herman and A.~Kuba (eds.), \emph{Advances in {D}iscrete {T}omography and
  its {A}pplications}, Applied and Numerical Harmonic Analysis, Birkh\"auser
  Boston Inc., Boston, MA, 2007.

\bibitem[Jan97]{quasicrystals-primer}
C.~Janot, \emph{Quasicrystals: {A} {P}rimer}, Oxford University Press, 1997.

\bibitem[KST95]{Klibanov-Sachs-Tikhonravov-1995}
M.~V. Klibanov, P.~E. Sacks, and A.~V. Tikhonravov, \emph{The phase retrieval
  problem}, Inverse Problems \textbf{11} (1995), no.~1, 1--28.

\bibitem[Lan02]{MR1878556}
S.~Lang, \emph{Algebra}, third ed., Graduate Texts in Mathematics, vol. 211,
  Springer-Verlag, New York, 2002.

\bibitem[LSS03]{Lemke-Skiena-Smith-2003}
P.~Lemke, S.~S. Skiena, and W.~D. Smith, \emph{Reconstructing sets from
  interpoint distances}, Discrete and computational geometry, Algorithms
  Combin., vol.~25, Springer, Berlin, 2003, pp.~507--631.

\bibitem[Moo00]{Moo00}
R.~V. Moody, \emph{{Model sets: A survey}}, 28\,pp., preprint, available at
  {http://arxiv.org/pdf/math/0002020}, 2000.

\bibitem[Nag93]{MR1232748}
W.~Nagel, \emph{Orientation-dependent chord length distributions characterize
  convex polygons}, J. Appl. Probab. \textbf{30} (1993), no.~3, 730--736.

\bibitem[RS82]{RoSey82}
J.~Rosenblatt and P.~D. Seymour, \emph{The structure of homometric sets}, SIAM
  J. Algebraic Discrete Methods \textbf{3} (1982), no.~3, 343--350.

\bibitem[Sch93]{MR1216521}
R.~Schneider, \emph{{Convex Bodies: The Brunn-Minkowski Theory}}, Encyclopedia
  of Mathematics and its Applications, vol.~44, Cambridge University Press,
  Cambridge, 1993.

\end{thebibliography}

\providecommand{\bysame}{\leavevmode\hbox to3em{\hrulefill}\thinspace}
\providecommand{\MR}{\relax\ifhmode\unskip\space\fi MR }
\providecommand{\MRhref}[2]{%
  \href{http://www.ams.org/mathscinet-getitem?mr=#1}{#2}
}
\providecommand{\href}[2]{#2}

\end{document}